\documentclass[12pt]{article}

\usepackage{amsfonts,amssymb,amsmath,amsthm,epsfig,euscript}
\usepackage{epstopdf}
\usepackage{float}

\newtheorem{theorem}{Theorem}

\newtheorem{corollary}{Corollary}
\newtheorem{proposition}{Proposition}

\theoremstyle{remark}
\newtheorem{remark}{Remark}

\theoremstyle{definition}
{
\newtheorem{definition}{Definition}

}

\long\def\symbolfootnote[#1]#2{\begingroup
\def\thefootnote{\fnsymbol{footnote}}\footnote[#1]{#2}\endgroup}

\newcommand{\des}[1][\sigma]{\mathrm{des}(#1)}

\newcommand{\red}[1][\sigma]{\mathrm{red}(#1)}

\newcommand{\md}[1][\sigma]{\mathrm{maxdrop}(#1)}
\newcommand{\sg}{\sigma}

\newcommand{\mmp}{\mathrm{mmp}}

\newcommand\bub{{\rm \mathsf{bubble}}}
\newcommand\stack{{\rm \mathsf{stack}}}
\newcommand\bsc{{\rm bsc}}
\newcommand\id{{\rm id}}
\newcommand\Des{{\rm DES}}

\newcommand{\fig}[2]{\begin{figure}[ht]
\centerline{\scalebox{.66}{\epsfig{file=#1.eps}}}
\caption{#2}
\label{fig:#1}
\end{figure}}

\title{The classification of 231-avoiding permutations by descents and maximum drop}

\author{
Matthew Hyatt \\
\small Department of Mathematics\\[-0.8ex]
\small University of California, San Diego\\[-0.8ex]
\small La Jolla, CA 92093-0112. USA\\[-0.8ex]
\small \texttt{mdhyatt@ucsd.edu}
\and
Jeffrey Remmel \\
\small Department of Mathematics\\[-0.8ex]
\small University of California, San Diego\\[-0.8ex]
\small La Jolla, CA 92093-0112. USA\\[-0.8ex]
\small \texttt{jremmel@ucsd.edu}
}

\date{
\small MR Subject Classifications: 05A15, 05E05}

\begin{document}
\maketitle

\begin{abstract}
\noindent We study the number of $231$-avoiding permutations with $j$-descents 
and maximum drop is less than or equal to $k$ which we denote by 
$a_{n,231,j}^{(k)}$. We show that 
$a_{n,231,j}^{(k)}$ also counts the number of Dyck paths of length $2n$ 
with $n-j$ peaks and height $\leq k+1$, and the number of 
ordered trees with $n$ edges, $j+1$ internal nodes, and of height $\leq k+1$.
We show that the generating functions for the $a_{n,231,j}^{(k)}$s with 
$k$ fixed satisfy a simple recursion.  We also use the combinatorics 
of ordered trees to prove  
new explicit formulas for $a_{n,231,j}^{(k)}$ as a function of $n$ in 
a number of special values of $j$ and $k$ and 
prove a simple recursion for the $a_{n,231,j}^{(k)}$s. 

\

\end{abstract}

\section{Introduction}

In \cite{ccdg}, Chung, Claesson, Dukes, and Graham studied generating 
functions for permutations according to the number of descents and 
the maximum drop. Here if $\sg = \sg_1 \ldots \sg_n$ is a permutation 
in the symmetric group $S_n$, then we say that 
$\sg$ has drop at $i$ if $\sg_i < i$ and $\sg$ has a descent 
at $i$ if $\sg_i > \sg_{i+1}$. MacMahon proved that 
the number of permutations with $k$ descents is equal to the 
number of permutations with $k$ drops. Let $[n]=\{1,2,\dots,n\}$.
We let 
$\Des(\sg) = \{i\in [n]:\sg_i > \sg_{i+1}\}$, $\des =|\Des(\sg)|$, and $\md[\sg] = \max\{i-\sg_i:i\in [n]\}$. We let $\mathcal{B}^{(k)}_n$ denote the set of 
permutations $\sg \in S_n$ such that 
$\md[\sg]\leq k$. 

There is another interpretation of $\mathcal{B}^{(k)}_n$ in terms of the classic bubble sort, which we denote by $\bub$. Let $\bsc(\sigma)=\min\{i:\bub^i(\sigma)=\id\}$, i.e. $\bsc(\sigma)$ is the minimum number of times that $\bub$ must be applied to $\sigma$ in order to reach the identity permutation. An inductive argument shows that $\bsc(\sigma)=\md$, thus $\mathcal{B}^{(k)}_n$ is the set of permutations in $S_n$ which can be sorted by applying $\bub$ $k$ times. Additionally, the permutations in $\mathcal{B}^{(k)}_n$ are in bijective correspondence with certain juggling sequences (see \cite{ccdg}).

Let  
\[A_n^{(k)}(x)=\sum_{\sigma\in\mathcal{B}_{n}^{(k)}}x^{\des}=\sum_{j=0}^{n-1} a_{n,j}^{(k)}x^j.\]
Note that for $k\geq n-1$, $\mathcal{B}_{n}^{(k)}=S_n$ and $A_{n}^{(k)}(x)$ becomes the classic Eulerian polynomial
\[A_n(x)=\sum_{\sigma\in S_n}x^{\des}
=\sum_{j=0}^{n-1}a_{n,j}x^j.\]
The coefficient $a_{n,j}$ is the number of permutations in $S_n$ with $j$ descents. These coefficients are called Eulerian numbers. For convenience we let $A_0(x)=1$.

In \cite{ccdg}, the authors show that for $n\geq 0$, $A_{n}^{(k)}(x)$ satisfies the following recurrence
\[A_{n+k+1}^{(k)}(x)=\sum_{i=1}^{k+1}{k+1 \choose i}(x-1)^{i-1}A_{n+k+1-i}^{(k)}(x),\]
with the initial conditions $A_{i}^{(k)}(x)=A_i(x)$ for $0\leq i\leq k$. This recurrence is equivalent to the following generating function formula
\[A^{(k)}(x,t)=\sum_{n\geq 0}A_{n}^{(k)}(x)t^n
=\frac{1+\sum_{r=1}^k\left(A_r(x)-\sum_{i=1}^r{k+1 \choose i}(x-1)^{i-1}A_{r-i}(x)\right)t^r}
{1-\sum_{i=1}^{k+1}{ k+1 \choose i}t^i(x-1)^{i-1}}.\]
They also find an explicit formula for $a_{n,j}^{(k)}$. Let
\[P_k(u)=\sum_{r=0}^{k}A_{k-r}(u^{k+1})(u^{k+1}-1)^r\sum_{i=r}^{k}{i\choose r}u^{-i},\]
and let
\[\sum_r\beta_k(r)u^r=P_k(u)\left(\frac{1-u^{k+1}}{1-u}\right)^{n-k},\]
then
\[A_{n}^{(k)}(x)=\sum_j\beta_k((k+1)j)x^{j}.\]
In other words, the coefficients $a_{n,j}^{(k)}$ of the polynomial $A_{n}^{(k)}(x)$ have the remarkable property that they are given by every $(k+1)$-st coefficient in the polynomial 
\[P_k(u)(1+u+u^2+\dots+u^k)^{n-k}.\]
For example setting $n=4$ and $k=2$ we have
\[
P_2(u)(1+u+u^2)^{4-2}=(1+u+2u^2+u^3+u^4)(1+u+u^2)^2\]
\[=1+3u+7u^2+10u^3+12u^4+10u^5+7u^6+3u^7+u^8.
\]
So the coefficients of $A_{4}^{(2)}(x)$ are given by every third coefficient in the above polynomial, that is
\[A_{4}^{(2)}(x)=1+10x+7x^2.\]




We now turn our attention to pattern avoidance. Given a sequence $\sg = \sg_1 \ldots \sg_n$ of distinct integers,
let $\red[\sg]$ be the permutation found by replacing the
$i$-th smallest integer that appears in $\sg$ by $i$.  For
example, if $\sg = 2754$, then $\red[\sg] = 1432$.  Given a
permutation $\tau=\tau_1 \ldots \tau_j$ in the symmetric group $S_j$, we say that the pattern $\tau$ {\em occurs} in $\sg = \sg_1 \ldots \sg_n \in S_n$ provided   there exists 
$1 \leq i_1 < \cdots < i_j \leq n$ such that 
$\red[\sg_{i_1} \ldots \sg_{i_j}] = \tau$.   We say 
that a permutation $\sg$ {\em avoids} the pattern $\tau$ if $\tau$ does not 
occur in $\sg$. Let $S_n(\tau)$ denote the set of permutations in $S_n$ 
which avoid $\tau$. In the theory of permutation patterns (see \cite{kit} for a comprehensive introduction to the area),  $\tau$ is called a {\em classical pattern}. We let $\mathcal{B}^{(k)}_{n,\tau} = S_n(\tau) \cap \mathcal{B}^{(k)}_n$. 
Thus $\mathcal{B}^{(k)}_{n,\tau}$ is the set of $\sg \in S_n$ such 
that $\md[\sg] \leq k$ and $\sg$ avoids $\tau$. For $k \geq 1$, 
we let $\mathcal{E}^{(k)}_{n,\tau}= \mathcal{B}^{(k)}_{n,\tau}-\mathcal{B}^{(k-1)}_{n,\tau}$. Thus $\mathcal{E}^{(k)}_{n,\tau}$ is the set 
of $\sg \in S_n$ such 
that $\md[\sg] = k$ and $\sg$ avoids $\tau$. We let 
\begin{eqnarray*}
A^{(k)}_{n,\tau}(x) &=& 
\sum_{\sg \in \mathcal{B}^{(k)}_{n,\tau}} x^{\des[\sg]} = 
\sum_{j=0}^{n-1}  a^{(k)}_{n,\tau,j} x^j\ \mbox{and} \\
E^{(k)}_{n,\tau}(x) &=& \sum_{\sg \in \mathcal{E}^{(k)}_{n,\tau}} 
x^{\des[\sg]} = \sum_{j=0}^{n-1}  e^{(k)}_{n,\tau,j} x^j.
\end{eqnarray*}
Let 
\begin{equation}\label{Adef} 
A^{(k)}_{\tau}(x,t) = 1 + \sum_{n\geq 1} A^{(k)}_{n,\tau}(x) t^n
\end{equation} 
and 
\begin{equation}\label{Adef} 
E^{(k)}_{\tau}(x,t) = 1 + \sum_{n\geq 1} E^{(k)}_{n,\tau}(x) t^n
\end{equation} 
Note that for $k \geq 1$, $E^{(k)}_{n,\tau}(x) = 
A^{(k)}_{n,\tau}(x)-A^{(k-1)}_{n,\tau}(x)$ so that 
$E^{(k)}_{n,\tau}(x,t) = A^{(k)}_{\tau}(x,t)-A^{(k-1)}_{\tau}(x,t)$. 

The main goal of this paper is to study the generating 
functions $A^{(k)}_{231}(x,t)$ and $E^{(k)}_{231}(x,t)$.

We remark that the set $\mathcal{B}^{(k)}_{n,231}$ can also be interpreted in terms of sorting algorithms. The 231-avoiding permutations are precisely the permutations which can be sorted by one application of the stack sort, which we denote by $\stack$ (see \cite{knuth1}). So $\mathcal{B}^{(k)}_{n,231}=\{\sigma\in S_n:\bub^k(\sigma)=\id,\text{ and } \stack(\sigma)=\id\}$, i.e. the permutations in $S_n$ which can be sorted by one stack sort, but require $k$ bubble sorts to be sorted.

Note that the only permutation $\sg \in S_n$ such that 
$\md[\sg]=0$ is the identity permutation $\sg = 123 \ldots n$ which 
is 231-avoiding. Thus $A^{(0)}_{n,231}(x) =1$ for all $n \geq 1$ so that 
\begin{equation}\label{A0}
A^{(0)}_{231}(x,t) = \frac{1}{1-t}.
\end{equation}
Our key theorem is to show that generating functions 
$A^{(k)}_{231}(x,t)$ for $k \geq 1$ satisfy the following 
simple recursion. 
\begin{theorem}\label{thm:main} For all $k \geq 1$,
\begin{equation}\label{Arec} 
A^{(k)}_{231}(x,t) = \frac{1}{1-t+tx-txA^{(k-1)}_{231}(x,t)}
\end{equation}
where 
\begin{equation*}
A^{(0)}_{231}(x,t) = \frac{1}{1-t}.
\end{equation*}
\end{theorem}

Theorem \ref{thm:main} allowed us to explicitly compute the 
values of $a_{n,231,j}^{(k)}$ and $e_{n,231,j}^{(k)}$ for small 
values of $j$, $k$, and $n$ which lead us to conjecture 
a number of simple formulas for $a_{n,231,j}^{(k)}$ and 
$e_{n,231,j}^{(k)}$ in certain special cases.  For example, we shall 
show that for all $n,j \geq 1$ and all $k \geq j$, 
$a_{n,231,j}^{(k)} = N(n,n-j) = \frac{1}{n}\binom{n}{j}\binom{n}{j+1}$
and  $e_{n,231,j}^{(j)} = \binom{n+j-1}{2j}$. Here the $N(n,j)$s are 
the Narayana numbers which count the number of 
Dyck paths of length 2n with $j$ peaks and the number of 
ordered trees $n$ edges and $k$ leaves. This suggested 
that the numbers  $a_{n,231,j}^{(k)}$ should also have 
natural combinatorial interpretations in terms of 
Dyck paths and ordered trees. In fact, we 
construct bijections to show that  
$a_{n,231,j}^{(k)}$ is the number of ordered trees with height 
less than or equal to $k+1$, $n$ edges, and $j+1$ internal nodes 
and is the number of Dyck paths of length $2n$ with $n-j$ peaks and 
height less than or equal to $k+1$.  

Kemp \cite{kemp} gave a general formula for the number 
of ordered trees with with height 
less than or equal to $k$, $n$ edges, and $j$ internal nodes so 
that we have a general formula for $a_{n,231,j}^{(k)}$. 
However, in many cases, Kemp's formula is unnecessarily complicated 
so that we use the combinatorics of ordered trees to 
derive an number of elegant formulas and recursions for 
the $a_{n,231,j}^{(k)}$s.  For example, 
we shall show that 
$$a_{n,231,j}^{(j-2)} =\frac{1}{n}\binom{n}{j}\binom{n}{j+1} - 
\binom{n+j-1}{2j} -(2j-3)\binom{n+j-2}{2j}$$
and that the $a_{n,231,j}^{(k)}$s satisfy the following 
simple recursion:
$$a_{n,231,j}^{(k)} = \sum_{i=0}^j a_{j,231,i}^{(k-1)} \binom{n+i}{2j}.$$

The outline of this paper is a follows. 
In section 2, we shall prove theorem \ref{thm:main} as well as 
provide simple proofs of the fact that 
$a_{n,231,1}^{(k)} = \binom{n}{2}$ for all $k \geq 1$ and $n \geq 2$, 
$a_{n,231,2}^{(k)} = \frac{(n-1)^2((n-1)^2-1)}{12}$ for all $k \geq 2$ and $n \geq 3$, and that $e_{n,231,2}^{(2)} = \binom{n+1}{4}$ for all $n \geq 3$. 
In section 3, we shall prove our alternative combinatorial 
interpretations of $a_{n,231,j}^{(k)}$ in terms of ordered trees and 
Dyck paths.  In section 4, we shall use the combinatorics of 
ordered trees to prove a number of formulas for   
$a_{n,231,j}^{(k)}$ and  $e_{n,231,j}^{(k)}$ as well as derive 
a new recursion for the  $a_{n,231,j}^{(k)}$s. Finally, in section 5, 
we shall briefly discuss some combinatorial identities that 
arise by comparing Kemp's formula and our formulas.

\section{Proof of Theorem 1}


In this section, we shall prove Theorem \ref{thm:main}.
The proof proceeds by  classifying the $231$-avoiding permutations 
$\sg = \sg_1 \ldots \sg_n$ by the position of $n$ 
in $\sg$.  Clearly each $\sg \in  S_n(231)$ has the structure 
pictured in Figure \ref{fig:BBasic2}. That is, in the graph of 
$\sg$, the elements to the left of $n$, $C_i(\sg)$, have 
the structure of a $231$-avoiding permutation, the elements 
to the right of $n$, $D_i(\sg)$, have the structure of a 
$231$-avoiding permutation, and all the elements in 
$C_i(\sg)$ lie below all the elements in 
$D_i(\sg)$.  Note that the number of $231$-avoiding 
permutations in $S_n$ is the Catalan number 
$C_n = \frac{1}{n+1} \binom{2n}{n}$ and the generating 
function for the $C_n$'s is given by 
\begin{equation}\label{Catalan}
C(t) = \sum_{n \geq 0} C_n t^n = \frac{1-\sqrt{1-4t}}{2t}=
\frac{2}{1+\sqrt{1-4t}}.
\end{equation}

\fig{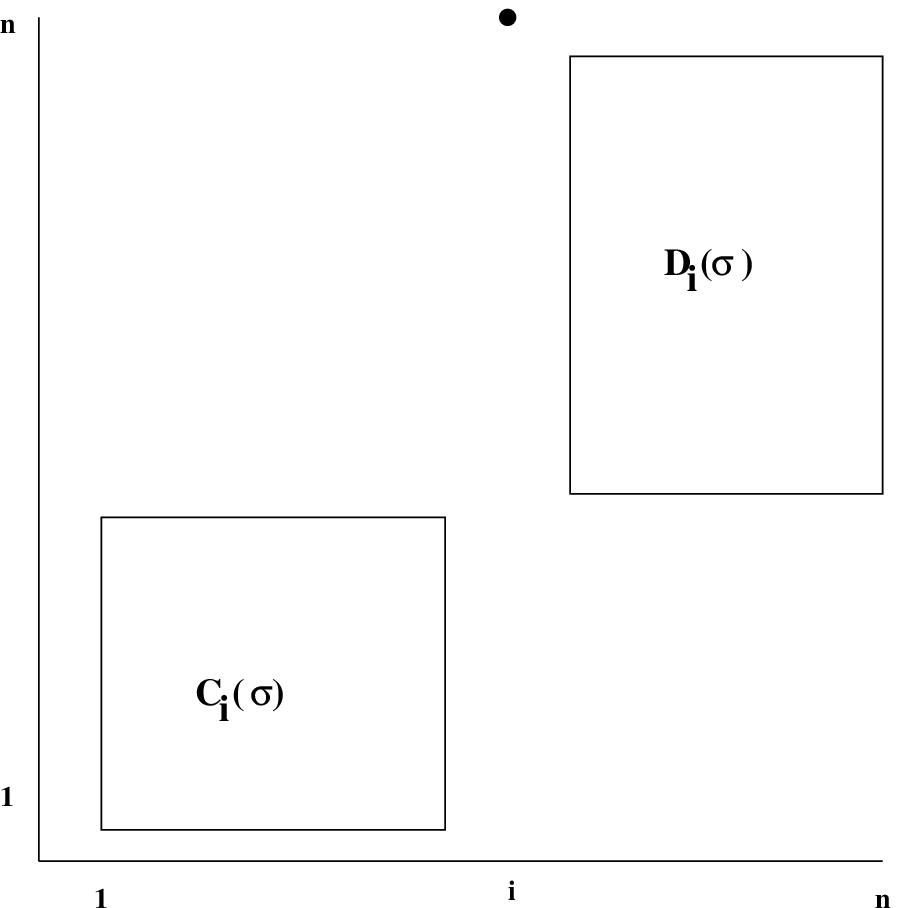}{The structure of $231$-avoiding permutations.}

Suppose that $k \geq 1$ and $n \geq 2$. Let 
$\mathcal{B}^{(k)}_{n,i,231}$ denote the set of $\sg \in 
\mathcal{B}^{(k)}_{n,231}$ such 
that $\sg_i =n$.  Clearly if $\sg \in \mathcal{B}^{(k)}_{n,i,231}$, 
then $C_i(\sg)$ must be a permutation in $S_{i-1}(231)$ such that 
$\md[C_i(\sg)] \leq k$. Similarly, $\tau = \red[D_i(\sg)]$ must 
be a permutation in $S_{n-i}(231)$ such that 
$\md[\tau] \leq k-1$. That is, we can consider $D_i(\sg)$ as 
a map from $\{(i+1, \ldots, i+(n-i)\}$ into $\{(i-1+1,\ldots, (i-1)+(n-i)\}$. 
Thus for $j=1, \ldots, n-i$, a drop 
$j -\tau_j$ in $\tau$ corresponds to a drop $i+j - (i-1+\tau_j)= i+j -\sg_{i+j}-1$ 
in $\sg$. Thus the drop at position $j$ in $\tau$ is one less 
than the drop at position $i+j$ in $\sg$. Now if 
$i \leq n-1$, then $\sg_i =n$ will start a descent in 
$\sg$. Thus the possible choices for $C_i(\sg)$ will contribute 
a factor of $A^{(k)}_{i-1,231}(x)$ to 
$\sum_{\sg \in \mathcal{B}^{(k)}_{n,i,231}} x^{\des[\sg]}$ and 
the possible choices for $D_i(\sg)$ will contribute 
a factor of $A^{(k-1)}_{n-i,231}(x)$ to 
$\sum_{\sg \in \mathcal{B}^{(k)}_{n,i,231}} x^{\des[\sg]}$. 
Thus the contribution of the permutations in  $\mathcal{B}^{(k)}_{n,i,231}$ 
to $A^{(k)}_{n,231}(x)$ is $x A^{(k)}_{i-1,231}(x) A^{(k-1)}_{n-i,231}(x)$.
Finally, it is easy to see that the contribution of  the permutations in  $\mathcal{B}^{(k)}_{n,n,231}$ 
to $A^{(k)}_{n,231}(x)$ is just $ A^{(k)}_{n-1,231}(x)$. It follows that 
for $n \geq 2$, 
\begin{equation}\label{eq:Arec}
A^{(k)}_{n,231}(x) = A^{(k)}_{n-1,231}(x) + x \sum_{i=1}^{n-1}
A^{(k)}_{i-1,231}(x) A^{(k-1)}_{n-i,231}(x).
\end{equation}
Note that $A^{(k)}_{1,231}(x)=1$ so that if we define 
$A^{(k)}_{0,231}(x) =1$, then (\ref{eq:Arec}) also holds for 
$n=1$. 
Multiplying both sides of (\ref{eq:Arec}) by $t^n$ and summing for 
$n \geq 1$, we see that for $k \geq 1$, 
\begin{eqnarray*}
A^{(k)}_{231}(x,t)-1 &=& tA^{(k)}_{231}(x,t) + 
t\sum_{n\geq 1} t^{n-1} x \sum_{i=1}^{n-1}
A^{(k)}_{i-1,231}(x) A^{(k-1)}_{n-i,231}(x) \\
&=& tA^{(k)}_{231}(x,t) +txA^{(k)}_{231}(x,t)(A^{(k-1)}_{231}(x,t)-1).
\end{eqnarray*}
Solving this equation for $A^{(k)}_{231}(x,t)$, we see that 
\begin{equation}\label{eq:Arecfinal}
A^{(k)}_{231}(x,t) = \frac{1}{1-t+xt-xtA^{(k-1)}_{231}(x,t)}
\end{equation}
which proves Theorem \ref{thm:main}.

One can use Mathematica to calculate the first few of the generating functions 
$A^{(k)}_{231}(x,t)$. 
\begin{eqnarray*}
A^{(0)}_{231}(x,t) &=& \frac{1}{1-t},\\
A^{(1)}_{231}(x,t) &=& \frac{1-t}{1-2t+(1-x)t^2},\\
A^{(2)}_{231}(x,t) &=& \frac{1-2t+(1-x)t^2}{1-3t+(3-2x)t^2-(1-x)^2t^3},\\
A^{(3)}_{231}(x,t) &=& \frac{1-3t+(3-2x)t^2-(1-x)^2t^3}{1-4t+3(x-2)t^2-2(2-3x_x^2)t^3+(1-x)^3t^4},\\
A^{(4)}_{231}(x,t) &=& \frac{1-4t+3(x-2)t^2-2(2-3x_x^2)t^3+(1-x)^3t^4}{1-5t+(10-4x)t^2-(10-12x+3x^2)t^3-(1-x)^2(2x-5)t^4-(1-x)^4t^5}.
\end{eqnarray*}

One can also use Mathematica to find the initial terms of 
the generating function $A^{(k)}_{231}(t,x)$. For example, 
we have computed that

\begin{align*}
&A^{(1)}_{231}(t,x)  \\ 
&= 1+t+(1+x)t^2+(1+3 x)t^3+\left(1+6 x+x^2\right) t^4+\left(1+10 x+5 x^2\right) t^5+\\
&\left(1+15 x+15 x^2+x^3\right) t^6+\left(1+21 x+35 x^2+7 x^3\right) t^7+\\
&\left(1+28 x+70 x^2+28 x^3+x^4\right) t^8+\left(1+36 x+126 x^2+84 x^3+9 x^4\right) t^9+\\
&\left(1+45 x+210 x^2+210 x^3+45 x^4+x^5\right)
t^{10}+\\
&\left(1+55 x+330 x^2+462 x^3+165 x^4+11 x^5\right) t^{11}+ \cdots 
\end{align*}

\begin{align*}
&A^{(2)}_{231}(t,x) \\
& = 1+t+(1+x) t^2+\left(1+3 x+x^2\right) t^3+\left(1+6 x+6 x^2\right) t^4+
\left(1+10 x+20 x^2+3 x^3\right) t^5+\\
&\left(1+15 x+50 x^2+22 x^3+x^4\right)
t^6+\left(1+21 x+105 x^2+91 x^3+15 x^4\right) t^7+\\
&\left(1+28 x+196 x^2+280 x^3+100 x^4+5 x^5\right) t^8+\\
&\left(1+36 x+336 x^2+714 x^3+444 x^4+65 x^5+x^6\right)
t^9+\\
&\left(1+45 x+540 x^2+1596 x^3+1530 x^4+441 x^5+28 x^6\right) t^{10}+ \\
&\left(1+55 x+825 x^2+3234 x^3+4422 x^4+2101 x^5+301 x^6+7 x^7\right) t^{11}+
\cdots 
\end{align*}

\begin{align*}
&A^{(3)}_{231}(t,x) \\
& =1+t+(1+x) t^2+\left(1+3 x+x^2\right) t^3+\left(1+6 x+6 x^2+x^3\right) t^4+\\
&\left(1+10 x+20 x^2+10 x^3\right) t^5+\left(1+15 x+50 x^2+50
x^3+6 x^4\right) t^6+\\
&\left(1+21 x+105 x^2+175 x^3+60 x^4+3 x^5\right) t^7+\\
&\left(1+28 x+196 x^2+490 x^3+325 x^4+53 x^5+x^6\right) t^8+\\
& \left(1+36 x+336
x^2+1176 x^3+1269 x^4+428 x^5+35 x^6\right) t^9+\\
& \left(1+45 x+540 x^2+2520 x^3+4005 x^4+2289 x^5+427 x^6+15 x^7\right) 
t^{10}+\\
&\left(1+55 x+825 x^2+4950
x^3+10857 x^4+9394 x^5+3122 x^6+316 x^7+5 x^8\right) t^{11}+ \cdots 
\end{align*}

\begin{align*}
&A^{(4)}_{231}(t,x) =\\
& 1+t+(1+x) t^2+\left(1+3 x+x^2\right) t^3+\left(1+6 x+6 x^2+x^3\right) t^4+\\
&\left(1+10 x+20 x^2+10 x^3+x^4\right) t^5+\left(1+15 x+50 x^2+50
x^3+15 x^4\right) t^6+\\
&\left(1+21 x+105 x^2+175 x^3+105 x^4+10 x^5\right) t^7+\\
&\left(1+28 x+196 x^2+490 x^3+490 x^4+130 x^5+6 x^6\right) t^8+\\
&\left(1+36
x+336 x^2+1176 x^3+1764 x^4+890 x^5+128 x^6+3 x^7\right) t^9+\\
&\left(1+45 x+540 x^2+2520 x^3+5292 x^4+4291 x^5+1246 x^6+105 x^7+x^8\right) t^{10}+\\
&\left(1+55
x+825 x^2+4950 x^3+13860 x^4+16401 x^5+7945 x^6+1435 x^7+70 x^8\right) t^{11}+
\end{align*}

The generating function for the number of 231-avoiding permutations $\sg$ 
with $\md[\sg] \leq k$ is $A^{(k)}_{231}(1,t)$ for any $k \geq 0$. 
These are easily computed using Theorem \ref{thm:main} and Mathematica. 
For example, we have have computed that 
\begin{eqnarray*}
A^{(0)}_{231}(1,t) &=& \frac{1}{1-t},\\
A^{(1)}_{231}(1,t) &=& \frac{1-t}{1-2 t},\\
A^{(2)}_{231}(1,t) &=&\frac{1-2 t}{1-3 t+t^2},\\
A^{(3)}_{231}(1,t) &=&\frac{1-3 t+t^2}{1-4 t+3 t^2},\\
A^{(4)}_{231}(1,t) &=&\frac{1-4 t+3 t^2}{1-5 t+6 t^2-t^3}, \ \mbox{and} \\
A^{(5)}_{231}(1,t) &=&\frac{1-5 t+6 t^2-t^3}{1-6 t+10 t^2-4 t^3}.
\end{eqnarray*}

These generating functions have recently turned up in a completely 
different context. In \cite{kitremtie}, 
Kitaev, Remmel, and Tiefenbruck  studied 
what they called {\em quadrant marked mesh patterns}. That is, let $\sigma = \sg_1 \ldots \sg_n$ be a permutation written in one-line notation. Then we will consider the 
graph of $\sg$, $G(\sg)$, to be the set of points $(i,\sg_i)$ for 
$i =1, \ldots, n$.  For example, the graph of the permutation 
$\sg = 471569283$ is pictured in Figure 
\ref{fig:basic}.  Then if we draw a coordinate system centered at a 
point $(i,\sg_i)$, we will be interested in  the points that 
lie in the four quadrants I, II, III, and IV of that 
coordinate system as pictured 
in Figure \ref{fig:basic}.  For any $a,b,c,d \in  
\mathbb{N}$ where $\mathbb{N} = \{0,1,2, \ldots \}$ is the set of 
natural numbers and any $\sg = \sg_1 \ldots \sg_n \in S_n$, 
we say that $\sg_i$ matches the 
quadrant marked mesh pattern $MMP(a,b,c,d)$ in $\sg$ if in $G(\sg)$  relative 
to the coordinate system which has the point $(i,\sg_i)$ as its  
origin,  there are 
$\geq a$ points in quadrant I, 
$\geq b$ points in quadrant II, $\geq c$ points in quadrant 
III, and $\geq d$ points in quadrant IV.  
For example, 
if $\sg = 471569283$, the point $\sg_4 =5$  matches 
the simple marked mesh pattern $MMP(2,1,2,1)$ since relative 
to the coordinate system with origin $(4,5)$,  
there are 3 points in $G(\sg)$ in quadrant I, 
1 point in $G(\sg)$ in quadrant II, 2 points in $G(\sg)$ in quadrant III, and 2 points in $G(\sg)$ in 
quadrant IV.  Note that if a coordinate 
in $MMP(a,b,c,d)$ is 0, then there is no condition imposed 
on the points in the corresponding quadrant.
In \cite{kitremtie}, the authors studied the generating 
functions 
\begin{equation} \label{Rabcd}
Q_{132}^{(a,b,c,d)}(t,x) = 1 + \sum_{n\geq 1} t^n  Q_{n,132}^{(a,b,c,d)}(x)
\end{equation}
where for  any $a,b,c,d \in  \mathbb{N}$, 
\begin{equation} \label{Rabcdn}
Q_{n,132}^{(a,b,c,d)}(x) = \sum_{\sg \in S_n(132)} x^{\mmp^{(a,b,c,d)}(\sg)}.
\end{equation}

\fig{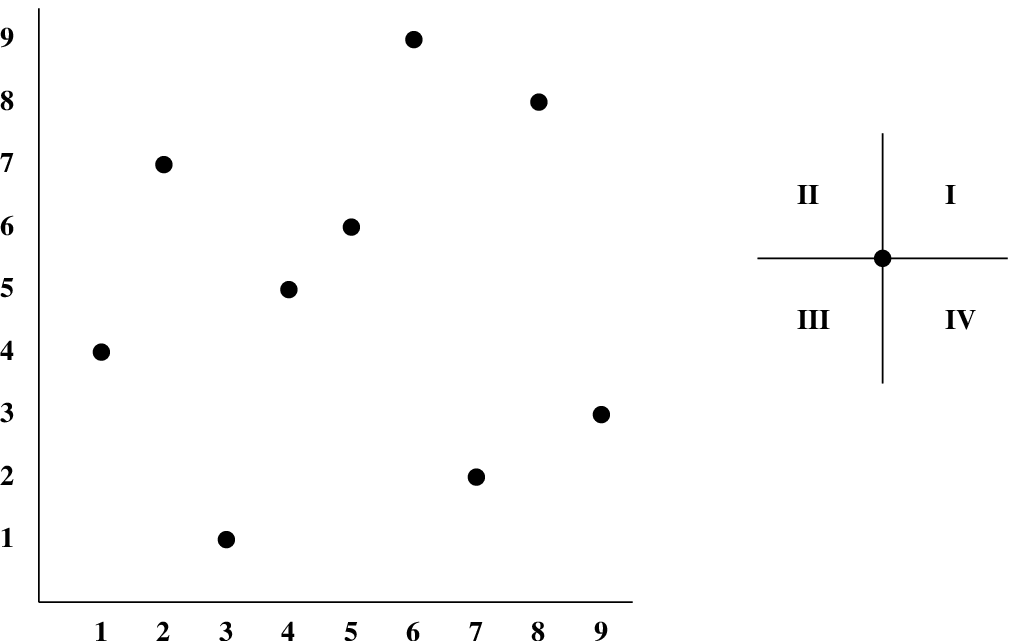}{The graph of $\sg = 471569283$.}

It turns out that $Q^{(k,0,0,0)}(t,0) = A^{(k-1)}(1,t)$ 
for  all $k \geq 2$ since it was shown in \cite{kitremtie} that 

$$Q_{132}^{(1,0,0,0)}(t,0) = \frac{1}{1-t}$$ and 
for $k > 1$, 
$$Q_{132}^{(k,0,0,0)}(t,0) = \frac{1}{1- tQ_{132}^{(k-1,0,0,0)}(t,0)}.$$
Thus the number of $231$-avoiding permutations with 
$\md[\sg] \leq k-1$ is the number of $132$-avoiding permutations 
which have no occurrences of the quadrant mesh pattern 
$MMP(k,0,0,0)$. In fact, one can use the recursions 
satisfied by $A^{(k-1)}_{n,231}(1)$ and $Q_{n,132}^{(k,0,0,0)}$ to 
give a bijective proof of this fact. 
It was also shown in \cite{kitremtie} that 
the number of permutations $\sg \in S_n(132)$ which have no occurrences of the quadrant mesh pattern 
$MMP(k,0,0,0)$ is also equal to the number of Dyck paths of length $2n$ 
such that all steps have height $\leq k$.

Many of the sequences 
$(A^{(k)}_{n,231}(1))_{n \geq 1}$ as well as many of 
the sequences $(a^{(k)}_{n,231,r})_{n \geq 1}$ appear in the OEIS. For example, the sequence $(A^{(2)}_{n,231}(1))_{n\geq 0}$ starts out with \\
$1,1,2,5,13,34,89,233,610,1597, \ldots.$. This is sequence 
A001519 in the OEIS. It immediately follows 
from the generating function
$$A^{(2)}_{231}(1,t) = \frac{1-2t}{1-3t+t^2}$$ 
that the numbers $A^{(2)}_{n,231}(1)$ satisfy the simple recursion 
that 
$$A^{(2)}_{n,231}(1) = 3A^{(2)}_{n-1,231}(1) - A^{(2)}_{n-2,231}(1)$$
with initial conditions that 
$A^{(2)}_{0,231}(1)=A^{(2)}_{1,231}(1)=1$.  The OEIS lists 
many combinatorial interpretations of these numbers including 
the number of permutations of $S_{n+1}$ which avoid 321 and 3412 and the 
number of ordered trees with $n+1$ edges and height of at most 3. 
In section 3, we shall establish a direction connection 
between ordered trees and the permutations 
in $\mathcal{B}^{(k)}_{n,231}$ which will explain this and many other 
formulas.

The sequence $(A^{(3)}_{n,231}(1))_{n\geq 0}$ starts out with 
$1,1,2,5,14,41,122,365,1094,3281, \ldots $. This is sequence 
A124302 in the OEIS. This sequence also has many combinatorial 
definitions including the number of set partitions of $[n] = \{1, \ldots, n\}$ 
of length $\leq 3$. 
The sequence $(A^{(4)}_{n,231}(1))_{n\geq 0}$ starts out with 
$1,1,2,5,14,42,131,417,1341,4434, \ldots.$ which is sequence 
A080937 in the OEIS. The sequence $(A^{(5)}_{n,231}(1))_{n\geq 0}$ 
starts out with \\
$1,1,2,5,14,42,132,428,1416,4744, \ldots $ which 
is sequence A024175 in the OEIS.

We end this section, with a few simple results which can 
be easily proved from (\ref{eq:Arec}).  These results will 
be generalized in the subsequent section when we consider 
a bijection between ordered trees and permutations 
$\sg$ with $\md[\sg]\leq k$. 
\begin{theorem}\label{thm:1} 
\begin{enumerate}
\item For all $n \geq 1$, $A^{(1)}_{n,231}(1) = 2^{n-1}$. 

\item For all $r \geq 1$ and $n \geq 2r$, 
$A^{(1)}_{n,231}(x)|_{x^r} = a^{(1)}_{n,231,r}= \binom{n}{2r}$. 
\end{enumerate}
\end{theorem}
\begin{proof}

Part (1) follows immediately form the fact that 
$A^{(1)}_{231}(1,t) = \frac{1-t}{1-2 t}$.  It is also easy 
to give a direct inductive proof of part (1). 
That is, for part (1), 
clearly the statement holds for $n =1$ since $A^{(1)}_{1,231}(1) =1$. 
But then  using the fact that 
$A^{(0)}_{n,231}(1) =1$ for all $n \geq 0$, we see that (\ref{eq:Arec}) 
implies 
\begin{eqnarray*}
A^{(1)}_{n,231}(1) &=& A^{(1)}_{n-1,231}(1) + \sum_{i=1}^{n-1} 
A^{(1)}_{i-1,231}(1) A^{(0)}_{n-i,231}(1)\\
&=&   A^{(1)}_{n-1,231}(1) + A^{(0)}_{n-1,231}(1) + \sum_{i=2}^{n-1} 
A^{(1)}_{i-1,231}(1)\\
&=& 2^{n-2} +1 + \sum_{i=2}^{n-1} 2^{i-2} = 2^{n-1}.
\end{eqnarray*}

In fact, we can directly construct all the elements 
 $\mathcal{B}^{(1)}_{n,231}$.  
We let  $[n] =\{1, \ldots, n\}$ and, if $1 \leq i < j\leq n$, we 
let $[i,j] = \{s \in [n]: i \leq s \leq j\}$ be the interval from 
$i$ to $j$. Let $\mathcal{P}([n])$ denote 
the set of all subsets of $[n]$ and $\mathcal{P}_e([n])$ denote the 
set of all elements of $\mathcal{P}([n])$ that have even cardinality. 
Clearly, the cardinality of $\mathcal{P}_e([n])$ is $2^{n-1}$. We define 
bijection  $\phi:\mathcal{P}_e([n]) \rightarrow \mathcal{B}^{(1)}_{n,231}$. 
We let $\phi(\emptyset) = 12\ldots n$.  Now if 
$S = \{s_1,s_2, \ldots, s_{2r-1},s_{2r}\} \in \mathcal{P}_e([n])$ 
where $1 \leq s_1 < s_ 2 < \cdots < s_{2r-1} < s_{2r} \leq n$, 
then we consider the intervals 
$I_j = [s_{2j-1},s_{2j}]$ for $j =1, \ldots, r$. We define 
$\phi(S) = \tau^S  = \tau_1^S \ldots \tau_n^S$ to be 
the permutation in $S_n$ such that $\tau_i^S =i $ if $i$ is 
not in one of the intervals $I_1, \ldots, I_r$, and 
$\tau^S_{s_{2j-1}} \ldots \tau^S_{s_{2j}} = 
s_{2j}s_{2j-1}(s_{2j-1}+1) \ldots (s_{2j}-1)$. 
For example, 
if $n=12$, and $S=\{1,3,6,8,10,12\}$, then 
$I_1 = [1,3]$, $I_2 =[6,8]$, and $I_3 =[10,12]$. Thus 
$$\tau^S =3~1~2~4~5~8~6~7~9~12~10~11.$$
Note on 
each of the intervals $I_j$, $\tau^S$ has maximum drop 1 so that 
$\md[\tau^S] \leq 1$ for all $S \in \mathcal{P}_e([n])$.  Moreover it easy 
to see that $\tau^S$ is 231-avoiding and that we can recover $S$ from 
$\tau^S$.  Thus $\phi$ is a one-to-one map from $\mathcal{P}_e([n])$ 
into $\mathcal{B}^{(1)}_{n,231}$. However, since we 
know that  $|\mathcal{P}_e([n])| =|\mathcal{B}^{(1)}_{n,231}|$, 
$\phi$ must also be a surjection. Thus $\phi$ is a bijection 
from $\mathcal{P}_e([n])$ 
onto $\mathcal{B}^{(1)}_{n,231}$.

Note that $\phi$ also has the property that for all 
$S \in \mathcal{P}_e([n])$, $\des[\phi(S)] = \frac{|S|}{2}$.  Thus 
it follows that the number of $\sg \in \mathcal{B}^{(1)}_{n,231}$ such 
that $\des =r$ equals the number of subsets $S$ of $[n]$ of 
size 2r.  That is, 
$A^{(1)}_{n,231}(x)|_{x^r} = \binom{n}{2r}$ which proves part (2). 

\end{proof}

In fact, our construction in Theorem \ref{thm:1} constructs 
all the possible elements of $S_n(231)$ with exactly one descent. 
That is, we claim that if $\sg \in S_n(231)$ and 
$\des[\sg]=1$, then $\md[\sg] =1$.  Suppose that 
$\sg = \sg_1 \ldots \sg_n \in S_n(231)$, $\des[\sg]=1$ and 
$\md[\sg]=k$ where $k \geq 2$.  Let $i$ be least element 
such that $i-\sg_i=k$. Thus $\sg_i = i-k$. Our choice of 
$i$ ensures that $\sg_{i-1} \geq i-k$ since we can have a drop of 
at most $k-1$ at position $i-1$. But since $\sg_i =i-k$, it must 
be the case that $\sg_{i-1} > i-k = \sg_i$.  Thus the only 
descent of $\sg$ must occur at position $i-1$. This means 
that $\sg_1 < \cdots < \sg_{i-1}$.  Since $k \geq 2$, there 
must be at least two elements in $\sg_1 \ldots \sg_{i-1}$ which 
are greater than or equal to $i-k$ which would mean that there is an 
occurrence of 231 in $\sg_1 \ldots \sg_i$.  Thus if 
$\sg \in S_n(231)$ and $\des[\sg] =1$, then it must be the 
case that $\md[\sg]=1$.   Thus the elements 
of the form $\phi(S)$ where $|S|=2$ consists of all 
the elements of $S_n(231)$ such that $\des[\sg]=1$.  It 
thus follows that we have the following theorem. 

\begin{theorem} \label{thm:2} For all $k \geq 1$ and $n \geq 2$, 
$$A^{(k)}_{n,231}(x)|_x = a^{(k)}_{n,231,1}=\binom{n}{2}.$$
\end{theorem}

The sequence $(a^{(2)}_{n,231,2})_{n\geq 3}$ starts out with 
$1,6,20,50,105,196,336,540, \ldots.$. This is sequence 
A002415 in the OEIS whose $n$-th term is $\frac{n^2 (n^2-1)}{12}$. These 
numbers are known as the $4$-dimensional pyramidal numbers. They have 
several combinatorial interpretations including the number of squares 
with corners in the $n \times n$ grid. We can prove the following 
general theorem about such numbers.

\begin{theorem} \label{thm:3}
For all $k \geq 2$ and $n \geq 3$, 
$A^{(k)}_{n,231}(x)|_{x^2} =  \frac{(n-1)^2 ((n-1)^2-1)}{12}$.
\end{theorem}
\begin{proof}
We proceed by induction on $n$.  For the base case, note that 
the only permutation $\sg \in S_n(231)$ with 2 descents 
is $321$ which has maximum drop 2.  It follows 
that $A^{(k)}_{3,231}(x)|_{x^2} =  1$ for all $k \geq 2$ so that 
our formulas hold for $n =3$. 

Next assume that $n > 3$ and our formula holds for all $m < n$. Then 
by (\ref{eq:Arec}),
\begin{equation}
A^{(k)}_{n,231}(x)|_{x^2} = A^{(k)}_{n-1,231}(x)|_{x^2} + \left(\sum_{i=1}^{n-1}
A^{(k)}_{i-1,231}(x) A^{(k-1)}_{n-i,231}(x)\right)|_{x}.
\end{equation}
By induction, $A^{(k)}_{n-1,231}(x)|_{x^2} = \frac{(n-2)^2((n-2)^2-1)}{12}$. 
Note that 
\begin{eqnarray*} 
\left(\sum_{i=1}^{n-1}
A^{(k)}_{i-1,231}(x) A^{(k-1)}_{n-i,231}(x)\right)|_{x} &=& 
\sum_{i=1}^{n-1}
A^{(k)}_{i-1,231}(x)|_{x} A^{(k-1)}_{n-i,231}(x)|_{x^0} +\\
&& \sum_{i=1}^{n-1}
A^{(k)}_{i-1,231}(x)|_{x^0} A^{(k-1)}_{n-i,231}(x)|_{x}.
\end{eqnarray*}
But we know that $A^{(k)}_{i-1,231}(x)|_{x^0}=1$ for 
all $k$ and $n$ since the only permutation with no descents is 
the identity. By Theorem \ref{thm:2}, we know 
that $A^{(k)}_{n,231}(x)|_x = \binom{n}{2}$ for all $k \geq 1$ and 
$n \geq 1$.  It follows that 
\begin{eqnarray*} 
A^{(k)}_{n,231}(x)|_{x^2}&=& \frac{(n-2)^2((n-2)^2-1)}{12} + \sum_{i=1}^{n-1} 
\binom{i-1}{2} + \sum_{i=1}^{n-1} 
\binom{n-i}{2} \\
&=& \frac{(n-2)^2((n-2)^2-1)}{12} +\binom{n-1}{3} + \binom{n}{3} \\
&=& \frac{(n-1)^2((n-1)^2-1)}{12}
\end{eqnarray*}
where the last equality can be checked in Mathematica.
\end{proof}

Theorems \ref{thm:1} and Theorem \ref{thm:3} automatically 
imply the following theorem. 

\begin{theorem}\label{thm:4}
For $ n\geq 3$, $E^{(2)}_{n,132}(x)|_{x^2} = \binom{n+1}{4}$.
\end{theorem}

\begin{proof}
By Theorem \ref{thm:1}, we know that 
$A^{(1)}_{n,132}(x)|_{x^2} = \binom{n}{4}$ and Theorem \ref{thm:3}, we know that 
$A^{(2)}_{n,132}(x)|_{x^2} = \frac{(n-1)^2(((n-1)^2)-1)}{12}$. 
Thus 
\begin{eqnarray*}
E^{(2)}_{n,132}(x)|_{x^2} &=& A^{(2)}_{n,132}(x)|_{x^2}  - A^{(1)}_{n,132}(x)|_{x^2} \\
&=& \frac{(n-1)^2(((n-1)^2)-1)}{12} - \binom{n}{4} = \binom{n+1}{4}.
\end{eqnarray*}

\end{proof}

\section{Ordered trees of bounded height}

In this section we show there is a bijective correspondence between permutations in $S_n(231)$ with a given maximum drop and a given number of descents, to a certain class of trees. An ordered tree is a rooted tree where the children of each vertex are ordered, so for example we can refer to the left-most child of a vertex. We use the convention of placing the root at the top of the tree. Micheli and Rossin show there is a bijection between $231$-avoiding permutations and ordered trees \cite{mr}. Here we show this same bijection also carries additional information about the descents and maximum drop size of $231$-avoiding permutations. The level of a vertex is the distance of the shortest path from that vertex to the root. The height of an ordered tree is the maximum of the levels of all vertices in the tree. An internal node is a vertex which has at least one child. Let $\mathcal{T}^{(k)}_{n,j}$ denote the set of all ordered trees having $n$ edges, height less than or equal to $k$, and $j$ internal nodes. Let $\mathcal{B}^{(k)}_{n,231,j}$ denote the set of permutations in $\sigma\in S_n(231)$ with $\des =j$ and $\md\leq k$, thus $\left|\mathcal{B}^{(k)}_{n,231,j}\right|=A^{(k)}_{n,231}(x)|_{x^j}=a^{(k)}_{n,231,j}$ (not to be confused with $\mathcal{B}^{(k)}_{n,i,231}$ in the proof of Theorem \ref{thm:main}).

\begin{theorem}\label{perms to trees}

There is a bijection $\phi:\mathcal{T}^{(k)}_{n,j}\rightarrow \mathcal{B}^{(k-1)}_{n,231,j-1}$, for all $n,k,j\geq 1$, thus 
\[a_{n,231,j}^{(k)}=\left|\mathcal{T}_{n,j+1}^{(k+1)}\right|.\]

In other words $a_{n,231,j}^{(k)}$ is equal to the number of ordered trees with $n$ edges, $j+1$ internal nodes, and height less than or equal to $k+1$.

\end{theorem}

\begin{proof}

Given $T\in\mathcal{T}_{n,j}^{(k)}$, label the edges by a postorder traversal. Read the labels by a preorder traversal to obtain a word $\sigma\in S_n$. Set $\phi(T)=\sigma$.

For example, consider the ordered tree $T$ in Figure \ref{fig:permstotreesexample}.
\begin{figure}[h]
\centering
\includegraphics[height=7cm]{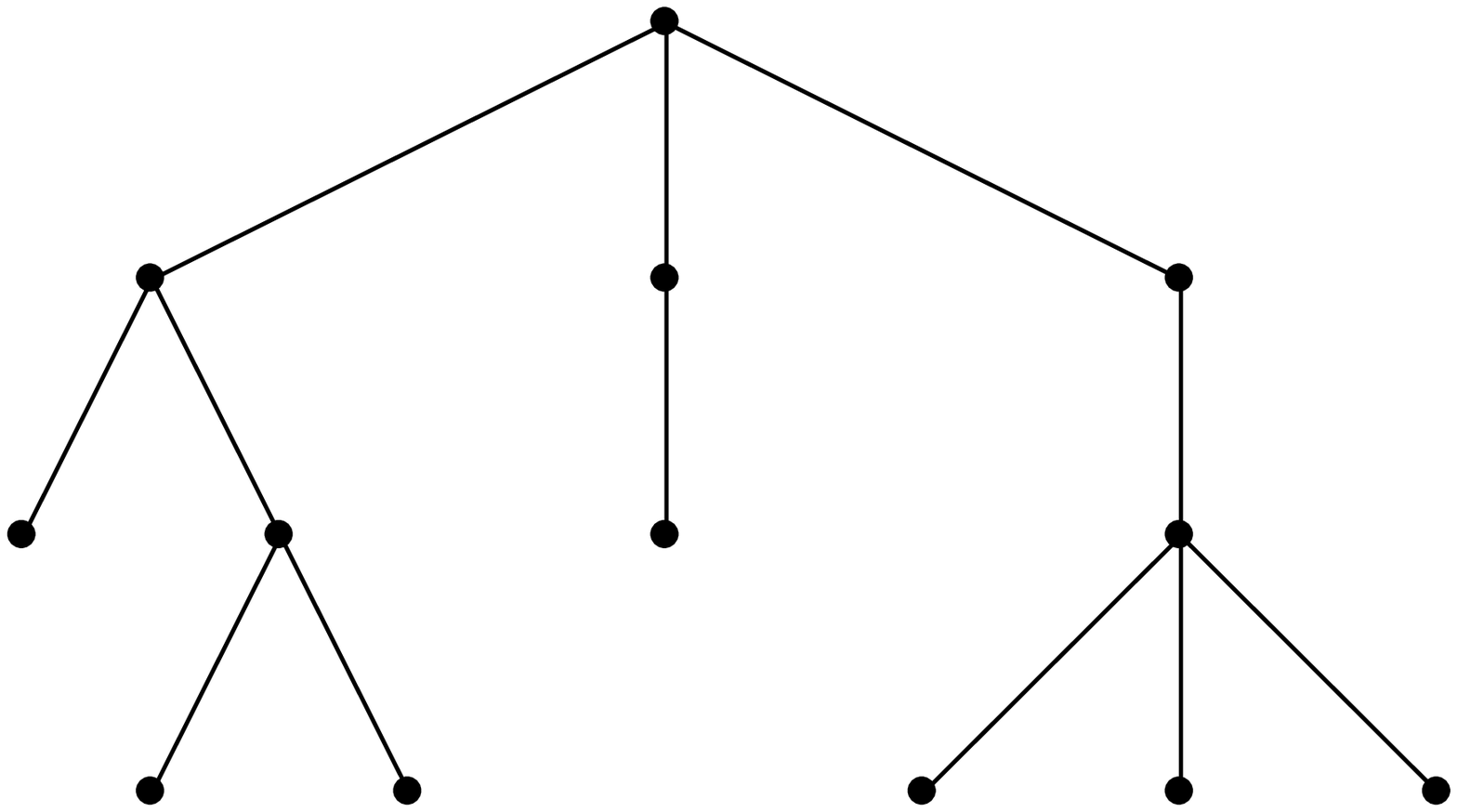}\includegraphics[height=7cm]{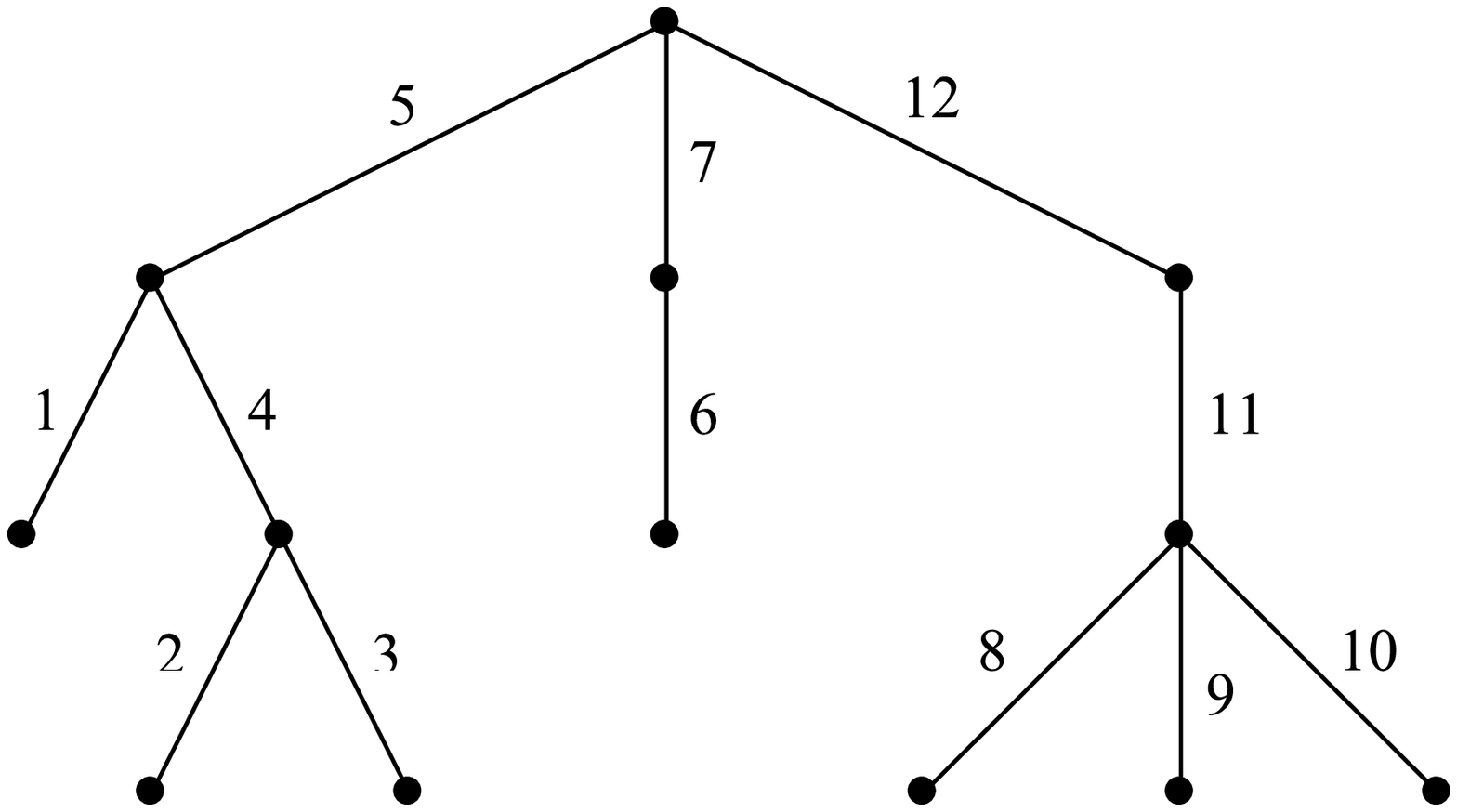}
\caption{An ordered tree $T$ (left), $T$ with edges labeled by a postorder traversal (right).}
\label{fig:permstotreesexample}
\end{figure}
When we read the labels by a preorder traversal, we obtain the permutation $\sigma$ (which we write in two-line notation)
\[\sigma=\left[\begin{array}{cccccccccccc}
1 & 2 & 3 & 4 & 5 & 6 & 7 & 8 & 9 & 10 & 11 & 12\\
5 & 1 & 4 & 2 & 3 & 7 & 6 & 12 & 11 & 8 & 9 & 10
\end{array}\right]\]

Since Micheli and Rossin showed that $\sigma\in S_n(231)$ and $\phi$ is a bijection (see \cite{mr}), our Theorem is proved if we can show that $\des =j-1$ and $\md \leq k-1$.

First we show that $\des=j-1$. Given any edge of $T$, let $\sigma_i$ be its label from the postorder traversal, and let $x$ be the vertex at the bottom of this edge. 

If $x$ is an internal node, then $\sigma_{i+1}$ is the label on the leftmost edge immediately below $x$. Since the labeling is done with a postorder traversal, we have $\sigma_i>\sigma_{i+1}$. 

If $x$ is not an internal node (i.e. a leaf), then there is a vertex $y$ with subtrees $Y_1$ and $Y_2$ such that $\sigma_i$ is a label on an edge of $Y_1$, $\sigma_{i+1}$ is a label on an edge of $Y_2$, and $Y_1$ is to the left of $Y_2$. It follows that $\sigma_i<\sigma_{i+1}$.

Since every vertex other than the root is at the bottom of a unique edge, $\sigma$ has $j-1$ descents.

Next we show that $\md\leq k-1$. Suppose $\sigma_i<i$, and let $x$ be the vertex at the bottom of the edge labeled $\sigma_i$.

If $x$ is an internal node, then $\sigma_i>\sigma_{i+1}$ as noted above, thus there is a larger drop size at position $i+1$ in $\sigma$. Since we want to find the maximum drop size, we need not consider the case that $x$ is an internal node.

Now assume that $x$ is not an internal node, and let $m$ be the level of $x$. On the path from $x$ to the root, there are $m$ (possibly empty) subtrees along the left side of the path, as in Figure \ref{fig:permstotreesheight}.
\begin{figure}[h]
\centering
\includegraphics[height=7.5cm]{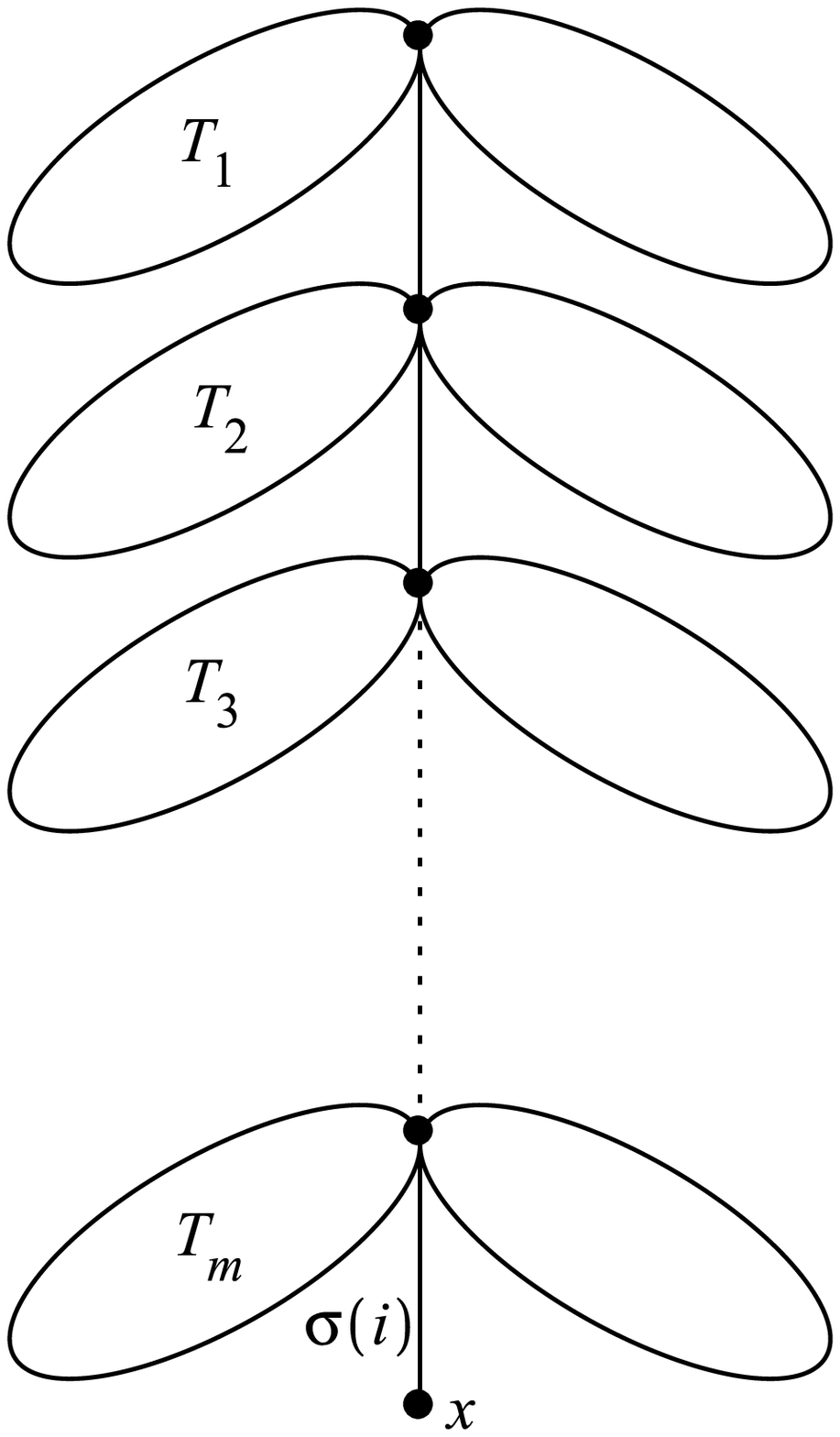}
\caption{An ordered tree with edge labeled $\sigma_i$ directly above a leaf $x$ at level $m$.}
\label{fig:permstotreesheight}
\end{figure}
Let $|T_r|$ denote the number of edges in a tree $T_r$. Then we have
\[|T_1|+|T_2|+\dots+|T_m|=\sigma_i-1,\]
since the edges in the subtrees $T_1,T_2,\dots,T_m$ are precisely the edges which precede the edge labeled $\sigma_i$ in the postorder traversal. The edges in the subtrees $T_1,T_2,\dots,T_m$ along with the edges in the path from $x$ to the root are precisely the edges which precede the edge labeled $\sigma_i$ in the preorder traversal, therefore
\[i=|T_1|+|T_2|+\dots+|T_m|+m=\sigma_i+m-1.\]
Thus $\sigma$ has a drop of size $i-\sigma_i=m-1$ at position $i$. Since $m\leq k$, we have $\md\leq k-1$.

\end{proof}

The set of trees in $\mathcal{T}_{n,j}^{(k)}$ are also in bijection with certain Dyck paths. A Dyck path of length $2n$ is a path in the plane that starts at the point $(0,0)$ and ends at the point $(2n,0)$. The path may consist only of up-steps $(1,1)$ and down-steps $(1,-1)$, and the path always stays on or above the $x$-axis. Let $\mathcal{D}_{2n}$ denote the set of Dyck paths of length $2n$. Next we describe a couple statistics for Dyck paths. The height of a Dyck path is the highest $y$-coordinate attained in the path. A peak is a point in a Dyck path which is immediately preceded by an up-step, and immediately followed by a down-step. Let $\mathcal{D}_{2n,j}^{(k)}$ denote the set of Dyck paths of length $2n$ with $j$ peaks and height less than or equal to $k$. The standard bijection from ordered trees to Dyck paths preserves height, and converts each leaf to a peak. A tree with $n$ edges and $j+1$ internal nodes has $n+1$ total nodes, thus $n-j$ leaves. From this it follows that $\left|\mathcal{T}_{n,j+1}^{(k)}\right|=\left|\mathcal{D}_{2n,n-j}^{(k)}\right|$.

Next we provide a direct bijection from permutations in $\mathcal{B}_{n,231,j}^{(k)}$ to Dyck paths in $\mathcal{D}_{2n,n-j}^{(k+1)}$. However, in subsequent sections of this paper we choose to use ordered trees to obtain enumeration results for the numbers $a_{n,231,j}^{(k)}$ (and $e_{n,231,j}^{(k)}$).

\begin{theorem}\label{perms to Dyck}

For all $n\geq 1$ and all $j,k\geq 0$, there is a bijection $\phi_n:\mathcal{B}^{(k)}_{n,231,j}\rightarrow \mathcal{D}^{(k+1)}_{2n,n-j}$. In other words, $a_{n,231,j}^{(k)}$ is equal to the number of Dyck paths of length $2n$ with $n-j$ peaks, and height less than or equal to $k+1$.

\end{theorem}

\begin{proof}

First we need to define the lifting of a path $P\in\mathcal{D}_{2n}$ to path $L(P)\in\mathcal{D}_{2n+2}$. Let $P=(p_1,\dots ,p_{2n})$ where each $p_i$ is either an up-step or a down-step. Then $L(P)$ is obtained from $P$ by appending an up-step at the start of $P$, and a down-step at the end of $P$. That is, $L(P)=((1,1),p_1,\dots ,p_{2n},(1,-1))$. An example is shown in Figure \ref{fig:map}. Also, if $P_1\in\mathcal{D}_{2n}$ and $P_2\in\mathcal{D}_{2k}$, then we let $P_1P_2\in\mathcal{D}_{2n+2k}$ denote the path which starts with $P_1$ followed by $P_2$.

\fig{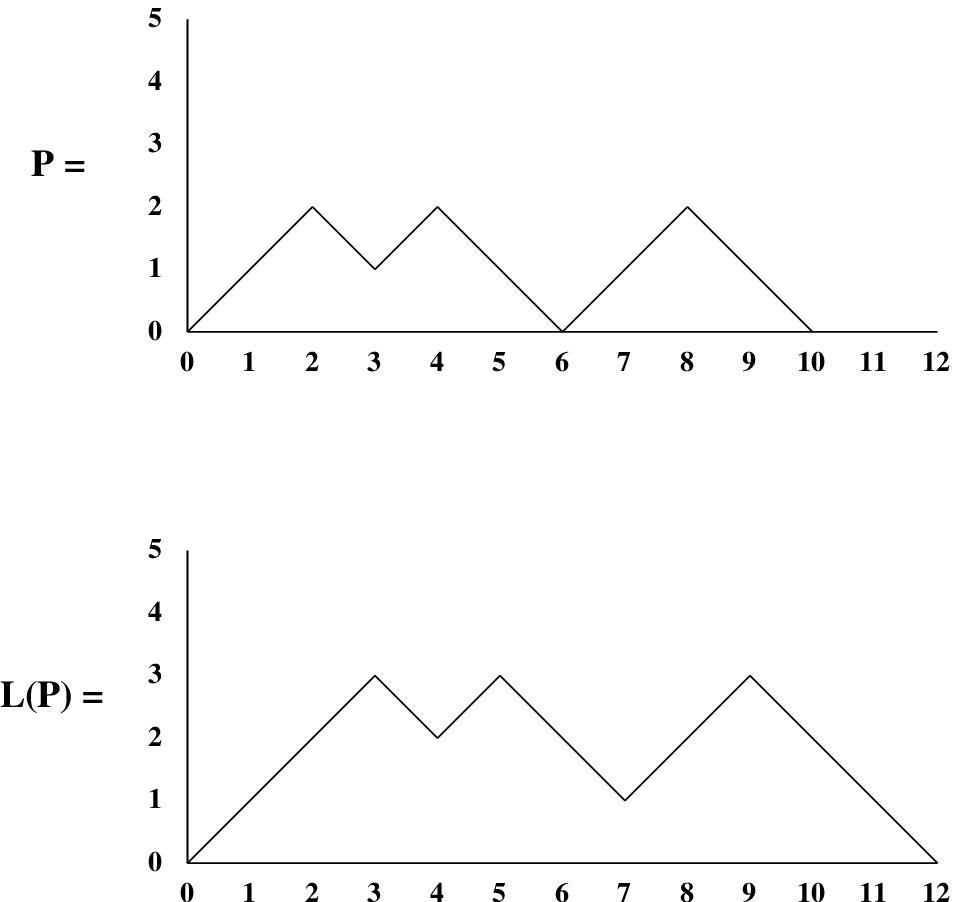}{The lifting of a Dyck path.}

For $n=1$, we simply define $\phi_1(\sigma)=((1,1),(1,-1))$, i.e. and up-step followed by a down-step. For $n>1$ we define $\phi_n$ recursively by cases as follows.

\ \\
{\bf Case 1.}  $\sg_n =n$. \\
Then $\phi_n(\sg) = P_1P_2$ where $P_1= \phi_{n-1}(\sg_1  \ldots \sg_{n-1})$ and $P_2 =  ((1,1),(1,-1))$. \\
\ \\
{\bf Case 2.}  $\sg_1 =n$. \\
Then $\phi_n(\sg) = L(\phi_{n-1}(\sg_2 \ldots \sg_{n}))$.\\
\ \\
{\bf Case 3.} $\sg_i = n$ where $1 < i < n$. In this case, 
$\phi_n(\sg) = P_1P_2$ where \\ 
$P_1 = \phi_{i-1}(\red[\sg_1 \ldots \sg_{i-1}])$ and 
$P_2 = L(\phi_{n-i}(\red[\sg_{i+1} \ldots \sg_n]))$.

\ \\
The first few values of this map a pictured in Figure \ref{fig:dyckpaths}, where $\sigma$ is on the left and $\phi_n(\sg)$ is on the right.

\begin{figure}[h]
\centering
\includegraphics[height=4cm]{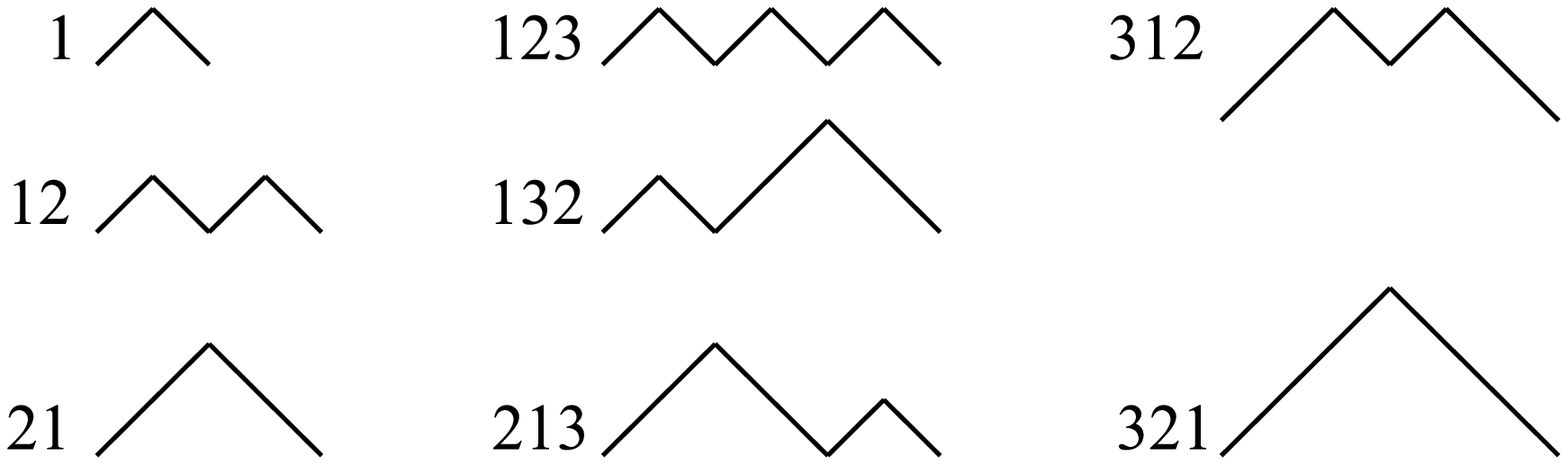}
\caption{Values of $\phi_n$ up to $n=3$.}
\label{fig:dyckpaths}
\end{figure}

Since $a_{n,231,j}^{(k)}=|\mathcal{T}_{n,j+1}^{(k+1)}|=|\mathcal{D}_{2n,n-j}^{(k+1)}|$, it suffices to show that $\phi_n$ well-defined and injective. We induct on $n$. The base case $n=1$ is obvious. Now let $n>1$ and assume the theorem holds for all $m<n$. Let $\sg=\sg_1\dots\sg_n\in\mathcal{B}_{n,231,j}^{(k)}$, and let $\sg_i=n$. Since $\sg$ avoid the pattern 231, we have $\sg_1\dots\sg_{i-1}\in\mathcal{B}_{i-1,231}^{(k)}$, and $\red[\sg_{i+1}\dots\sg_{n}]\in\mathcal{B}_{n-i,231}^{(k-1)}$ (see also the proof of Theorem \ref{thm:main}). To show $\phi_n$ is well-defined, we consider the three cases for $i$ in the definition of $\phi_n$.

\ \\
{\bf Case 1.} $\sg_n =n$.\\
In this case, $\phi_n(\sigma)=P_1P_2$ where $P_1= \phi_{n-1}(\sg_1  \ldots \sg_{n-1})$ and $P_2=((1,1),(1,-1))$. Since $n-1\notin\Des(\sg)$, we have $\sg_1  \ldots \sg_{n-1}\in\mathcal{B}_{n-1,231,j}^{(k)}$. By the inductive hypothesis we have $P_1\in\mathcal{D}_{2n-2,n-j-1}^{(k+1)}$. Since appending $P_2$ to $P_1$ increases the length by two, increases the number of peaks by one, and does not change the height, it follows that $P_1P_2\in\mathcal{D}_{n,n-j}^{(k+1)}$.

\ \\
{\bf Case 2.} $\sg_1 =n$.\\
In this case,  $\phi_n(\sg) = L(\phi_{n-1}(\sg_2 \ldots \sg_{n}))$. Since $1\in\Des(\sg)$, we have $\sg_2 \ldots \sg_{n}\in\mathcal{B}_{n-1,231,j-1}^{(k-1)}$. By induction, $\phi_{n-1}(\sg_2 \ldots \sg_{n})\in\mathcal{D}_{2n-2,n-j}^{(k)}$. Since lifting a path increases the height by one, increasing the length by two, and adds no peaks, it follows that 
 $L(\phi_{n-1}(\sg_2 \ldots \sg_{n}))\in\mathcal{D}^{(k+1)}_{2n,n-j}$.

\ \\
{\bf Case 3.} $\sg_i =n$ where $1<i<n$.\\
In this case, $\phi_n(\sg) = P_1P_2$ where $P_1 = \phi_{i-1}(\red[\sg_1 \ldots \sg_{i-1}])$ and \\
$P_2 = L(\phi_{n-i}(\red[\sg_{i+1} \ldots \sg_n]))$. Note that $\red[\sg_1\ldots \sg_{i-1}]=\sg_1\ldots \sg_{i-1}$. Since $i\in\Des(\sg)$ we have $\sg_1 \ldots \sg_{i-1}\in\mathcal{B}_{i-1,231,j_1}^{(k)}$, and $\red[\sg_{i+1}\ldots \sg_n]\in\mathcal{B}_{n-i,231,j_2}^{(k-1)}$, where $j_1+j_2+1=j$. Then $\phi_{n-i}(\red[\sg_{i+1}\ldots \sg_n])\in\mathcal{D}_{2n-2i,n-i-j_2}^{(k)}$, and 
$P_2= L(\phi_{n-i}(\red[\sg_{i+1}\ldots \sg_n]))\in\mathcal{D}_{2n-2i+2,n-i-j_2}^{(k+1)}$. Also, $P_1\in\mathcal{D}_{2i-2,i-1-j_1}^{(k+1)}$. It follows that $P_1P_2\in\mathcal{D}_{2n,n-j_1-j_2-1}^{(k+1)}$ as desired since $n-j_1-j_2-1=n-j$. 

\ \\
This proves $\phi_n$ is well-defined.\\

To prove injectivity let $\sg,\pi\in\mathcal{B}_{n,231,j}^{(k)}$, and suppose $\sigma\neq\pi$. If $\sg_i=\pi_i=n$ for some $i$, then for each case that $i$ falls into in the definition of $\phi_n$, one can easily use the inductive hypothesis to prove that $\phi_n(\sg)\neq\phi_n(\pi)$. So assume $\sg_{i_1}=n$ and $\pi_{i_2}=n$ where $i_1\neq i_2$. We consider the three possible combinations for $i_1$ and $i_2$ in the definition of $\phi_n$.

\ \\
{\bf Case I.} $\sg_n =n$ and $\pi_1=n$.\\
Since $\phi_n(\pi)=L(\phi_{n-1}(\pi_2 \ldots \pi_{n}))$, it follows that the second to last step of $\phi_n(\pi)$ is a down-step. Then $\phi_n(\sg)\neq\phi_n(\pi)$ since the second to last step of $\phi_n(\sg)$ is an up-step.

\ \\
{\bf Case II.} $\sg_n =n$ and $\pi_{i_2}=n$ where $1<i_2<n$.\\
Since $\phi_n(\pi)=P_1P_2$ where $P_2 = L(\phi_{n-i_2}(\red[\pi_{i_2+1}\ldots \pi_n]))$, we again have that the second to last step of $\phi_n(\pi)$ is a down-step, whereas the second to last step of $\phi_n(\sg)$ is an up-step. Thus $\phi_n(\sg)\neq\phi_n(\pi)$.

\ \\
{\bf Case III.} $\sg_1 =n$ and $\pi_{i_2}=n$ where $1<i_2<n$.\\
In this case we note that $\phi_n(\sg)=L(\phi_{n-1}(\sg_2 \ldots \sg_{n}))$, so that the only points where $\phi_n(\sg)$ touches the $x$-axis are at $(0,0)$ and $(2n,0)$. In contrast, $\phi_n(\pi)=P_1P_2$ where $P_1 = \phi_{i_2-1}(\red[\pi_1 \ldots \pi_{i_2-1}])$, so that $\phi_n(\pi)$ must touch the $x$-axis at the point $(2i_2-2,0)$. Thus $\phi_n(\sg)\neq\phi_n(\pi)$.

\end{proof}

\section{Recursions and closed form expressions for $a_{n,231,j}^{(k)}$}\label{recursions and closed}

In this section we prove some recursions and closed form expressions for $a_{n,231,j}^{(k)}$ and $e_{n,231,j}^{(k)}$. For certain cases of the values of $n,j,k$, we can find nice closed form expressions for these numbers. For the general case, it turns out that there is a closed form expression due to Kemp (see \cite {kemp}) for a class of trees very closely related to $\mathcal{T}_{n,j}^{(k)}$. This formula can easily be translated to a closed form expression for $a_{n,231,j}^{(k)}$. We also find a recurrence for $a_{n,231,j}^{(k)}$. We conclude this section by showing that this recurrence leads to closed form expression for $a_{n,231,j}^{(k)}$ which looks completely different from the formula due to Kemp.

\begin{theorem}[{\cite[Theorem 1]{kemp}}]\label{closed general trees}
Let $h_k(n,j)$ be the number of ordered trees with $n$ nodes, $j$ leaves, and height \footnote{In \cite{kemp}, the author uses the convention that the root is a vertex at level one, so we have translated this result to coincide with our definition of height.} less than or equal to $k-1$. Then $h_k(1,j)=\delta_{j,1}$, $h_1(n,j)=\delta_{n,j}\delta_{n,1}$ where $\delta_{n,j}$ is Kronecker's symbol. For $k\geq 2$ and $n\geq 2$
\[h_k(n,j)=N(n-1,j)-\left[Q_1(n,j,k)-2Q_0(n,j,k)+Q_{-1}(n,j,k)\right],\]
where
\[Q_a(n,j,k)=\sum_{s\geq 1} {n-s(k-1)-2 \choose j+s+a-1}{n+s(k-1)-2 \choose j-s-a-1},\]
and $N(n,j)$ are the Narayana numbers given by
\[N(n,j)=\frac{1}{n}{n \choose j}{n \choose j-1}.\]
\end{theorem} 

\begin{corollary}\label{closed general perms}

For all $n\geq 1$, and $j,k\geq 0$ we have
\[a_{n,231,j}^{(k)}=h_{k+2}(n+1,n-j).\]

\end{corollary}

\begin{proof}

A tree with $n$ edges has $n+1$ nodes. An ordered tree with $n$ edges and $j+1$ internal nodes has $n+1-(j+1)=n-j$ leaves. Thus from Theorem \ref{perms to trees} we have
\[a_{n,231,j}^{(k)}=\left|\mathcal{T}_{n,j+1}^{(k+1)}\right|=h_{k+2}(n+1,n-j).\]

\end{proof}

The Narayana numbers $N(n,j)$ appear in several combinatorial problems (see A001263 in the OIES \cite{oeis}).
One interpretation is that $N(n,j)$ is equal to the number of Dyck paths of length $2n$ with $j$ peaks. Another interpretation is that $N(n,j)$ is equal to the number of ordered trees with $n$ edges and $j$ leaves. Next we show that $a_{n,231,j}^{(k)}$ reduces to a Narayana number whenever $k\geq j$, extending the results from Theorem \ref{thm:2} and Theorem \ref{thm:3}.

\begin{corollary}\label{closed perms large k}

For all $n,j\geq 1$, and for all $k\geq j$ we have
\[a_{n,231,j}^{(k)}=N(n,n-j)=\frac{1}{n}{n\choose j}{n\choose j+1}.\]

\end{corollary}

\begin{proof}

An ordered tree with $n$ edges and $j+1$ internal nodes has height less than or equal to $j+1$, and $n-j$ leaves. Thus whenever $k\geq j$ we have
\[a_{n,231,j}^{(k)}=\left|\mathcal{T}_{n,j+1}^{(k+1)}\right|=N(n,n-j)=\frac{1}{n}{n\choose n-j}{n\choose n-j-1}
=\frac{1}{n}{n\choose j}{n\choose j+1}.\]

Note that in general $N(n,j)=N(n,n-j+1)$, i.e. the Narayana numbers are symmetric, and this follows from the symmetry of the binomial coefficients.

\end{proof}

In particular, $a_{n,231,1}^{(k)}={n\choose 2}$ for $k\geq 1$, and  $a_{n,231,2}^{(k)}=\frac{1}{n}{n\choose 2}{n\choose 3}=\frac{(n-1)^2((n-1)^2-1)}{12}$ for $k\geq 2$, as expected from Theorem \ref{thm:2} and Theorem \ref{thm:3}.

\begin{remark}

Corollary \ref{closed perms large k} also follows from Corollary \ref{closed general perms} and Theorem \ref{closed general trees} by noting that $Q_a(n+1,n-j,k+2)=0$ for $a=-1,0,1$ whenever
\[n+1-(k+1)-2< n-j+1+a-1.\]
And this inequality holds whenever $k\geq j$. Thus for $k\geq j$ we have
\[a_{n,231,j}^{(k)}=h_{k+2}(n+1,n-j)=N(n,n-j).\]

\end{remark}

Let $\mathcal{E}^{(k)}_{n,231,j}$ be the set of permutations $\sigma\in S_n(231)$ with $\des =j$ and $\md =k$. Thus $\mathcal{E}^{(k)}_{n,231,j}=\mathcal{B}^{(k)}_{n,231,j}-\mathcal{B}^{(k-1)}_{n,231,j}$, and $\left|\mathcal{E}^{(k)}_{n,231,j}\right|=e_{n,231,j}^{(k)}=a_{n,231,j}^{(k)}-a_{n,231,j}^{(k-1)}$. We can interpret $e_{n,231,j}^{(k)}$ as the number of ordered trees with $n$ edges, $j+1$ internal nodes, and height \textit{equal} to $k+1$. For any ordered tree, the number of internal nodes is always less than or equal to its height. So $e_{n,231,j}^{(k)}=0$ if $k\geq j+1$. Using the tree interpretation, we will directly compute $e_{n,231,j}^{(j)}$ by relating such trees to a certain set of weak compositions.

\begin{definition}\label{weak comp def}

Let $W_n(i,j)$ be the set of weak compositions $(p_1,p_2\dots ,p_i;m_1,m_2\dots ,m_j)$ in $\mathbb{N}^{i+j}$ such that (i) $p_r\geq 1$ for $r=1,2,\dots i$,
(ii) $m_r\geq 0$ for $r=1,2,\dots j$, and 
(ii) $\left(\sum_{r=1}^{i}p_r\right)+\left(\sum_{r=1}^{j}m_r\right)=n$. 
In other words, $W_n(i,j)$ is the set of weak compositions of $n$ with $i+j$ parts where the first $i$ parts are positive.

\end{definition}

\begin{proposition}\label{weak comp size}

For all $n,i,j\geq 0$ we have $|W_n(i,j)|={n+j-1\choose i+j-1}$.

\end{proposition}

\begin{proof}

Let $p_r'=p_r-1$ for $r=1,2,\dots ,i$. Then
\[(p_1,p_2,\dots ,p_i;m_1,m_2,\dots ,m_j)\in W_n(i,j)\]
if and only if
\[(p_1',p_2'\dots ,p_i';m_1,m_2,\dots ,m_j)\in W_{n-i}(0,i+j).\]
$W_n(0,k)$ is simply the number of weak compositions of $n$ into $k$ parts, and $|W_n(0,k)|={n+k-1\choose k-1}$ (see \cite{stanleyec1}). Thus
\[|W_n(i,j)|=|W_{n-i}(0,i+j)|={n+j-1\choose i+j-1}.\]

\end{proof}

\begin{theorem}\label{exact k=j}

For all $j\geq 1$ we have
\[e_{n,231,j}^{(j)}={n+j-1\choose 2j}.\]

Consequently,
\[a_{n,231,j}^{(j-1)}=\frac{1}{n}{n\choose j}{n\choose j+1}-{n+j-1\choose 2j}.\]

\end{theorem}

\begin{proof}

An ordered tree with $j+1$ internal nodes and height equal to $j+1$ must be a tree of the form shown in Figure \ref{fig:exact1} 
\begin{figure}[h]
\centering
\includegraphics[height=7.5cm]{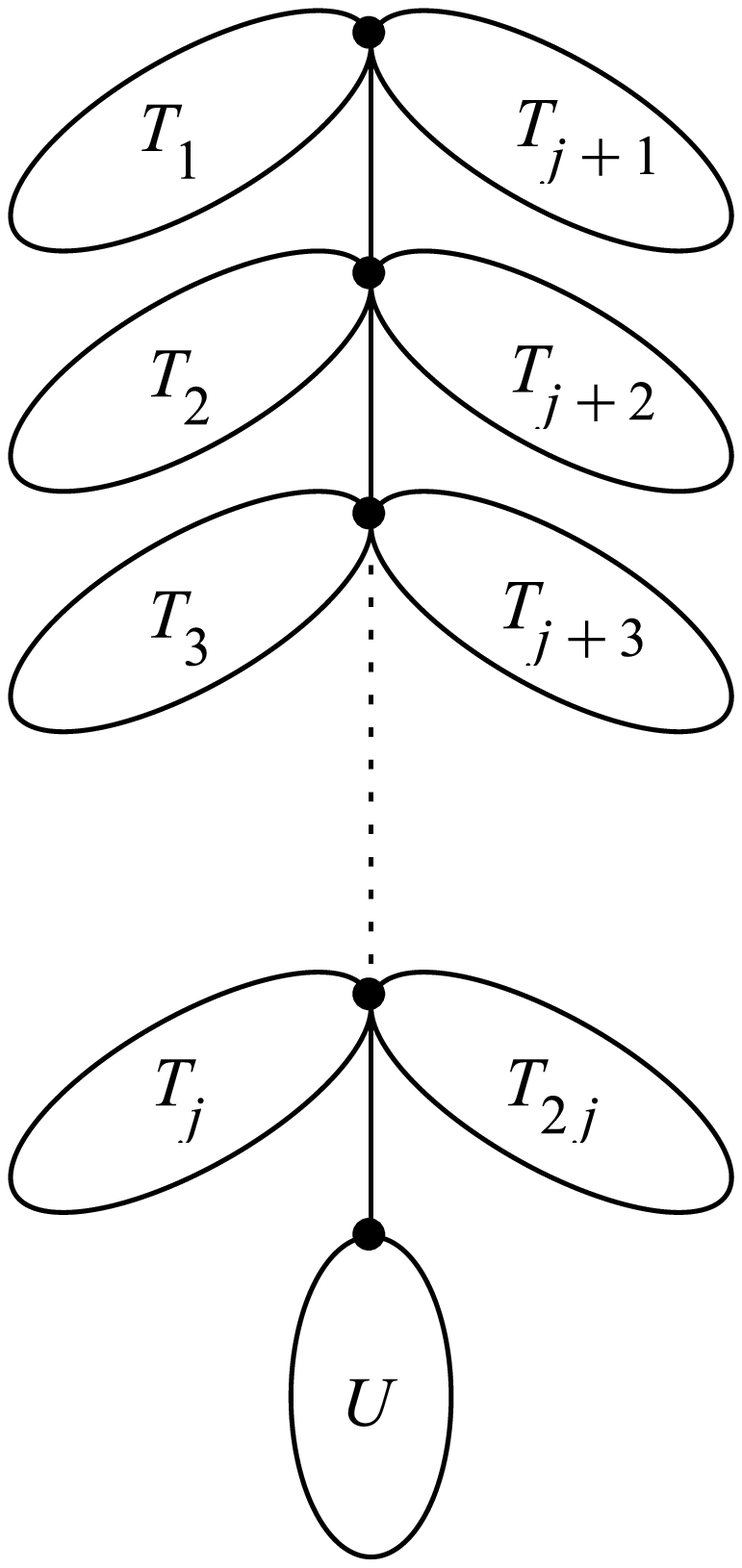}
\caption{An ordered tree with $j+1$ internal nodes and height equal to $j+1$.}
\label{fig:exact1}
\end{figure}
where each subtree $T_r$ has only one internal node (its root), and $|T_r|\geq 0$ for $r=1,2,\dots,2j$. The subtree $U$ must also have only one internal node, but $|U|\geq 1$ so that the whole tree has height $j+1$. In other words, $T_r\in\mathcal{T}_{m_r,1}^{(1)}$ with $m_r\geq 0$ for $r=1,2,\dots,2j$, and $U\in\mathcal{T}_{p,1}^{(1)}$ with $p\geq 1$. Note that $|\mathcal{T}_{m,1}^{(1)}|=1$ for all $m\geq 0$, so every tuple $(p;m_1,m_2\dots,m_{2j})\in W_{n-j}(1,2j)$ 
corresponds to a unique tree with $n$ edges, $j+1$ internal nodes, and height equal to $j+1$. Thus by Proposition \ref{weak comp size} we have
\[e_{n,231,j}^{(j)}=\left|W_{n-j}(1,2j)\right|={n+j-1\choose 2j}.\]

\end{proof}

In particular, $e_{n,231,2}^{(2)}={n+1\choose 4}$ as expected from Theorem \ref{thm:4}.

Since $e_{n,231,j}^{(j)}$ is equal to a binomial coefficient, we also provide a direct bijection between such sets and permutations.

\begin{proposition}\label{bijective exact k=j}
For all $j\geq 1$ there is a bijection 
\[\phi:{[n+j-1]\choose 2j}\rightarrow\mathcal{E}^{(j)}_{n,231,j},\]
where ${[n]\choose j}$ is the set of $j$-element subsets of $[n]$.
\end{proposition}

\begin{proof}

Let $S=\{a_1,a_2,\dots,a_{2j}\}\in {[n+j-1]\choose 2j}$ with $a_1<a_2<\dots<a_{2j}$. We construct from $S$ a sequence of nested intervals as follows. Let $I_1^S=[a_1,b_1]$ where $b_1=\max\{S\cap [n]\}$. Then for $m=2,3,\dots ,j$, let $I_m^S=[a_m,b_m]$ where
\[b_m=\begin{cases} b_{m-1} & \text{ if }n+m-1\in S \\ 
\max\{[1,b_{m-1}-1]\cap S\} & \text{ otherwise} \end{cases}.\]
By construction we have $a_1<a_2<\dots <a_j<b_j\leq b_{j-1}\leq b_{j-2}\leq \dots \leq b_1$. Next define a map $c_m^S$ which acts on permutations by cyclically rotating the letters in positions $a_m,a_m+1,\dots ,b_m$. More precisely, given $\pi\in S_n$ let $c_m^S(\pi)=\tau_1\dots \tau_n$ where $\tau_i=\pi_i$ if $i\notin I_m^S$, and $\tau_{a_m} \tau_{a_m+1} \dots \tau_{b_m}=\pi_{b_m} \pi_{a_m} \pi_{a_m+1} \dots \pi_{b_m-1}$. We then define
\[\phi(S)=c_j^S\circ c_{j-1}^S\circ\dots\circ c_1^S(\id).\]

For example let $n=7$, $j=3$, and let $S=\{1,3,4,5,6,8\}\in {[9] \choose 6}$. Then $b_1=\max\{S\cap [7]\}=6$, so $I_1^S=[1,6]$. Next we find $b_2$. Since $n+2-1=8\in S$, we set $b_2=b_1=6$ and $I_2^S=[3,6]$. Next we find $b_3$. Since $n+3-1=9\notin S$, we set $b_3=\max\{[1,b_2-1]\cap S\}=5$ and $I_3^S=[4,5]$. We can visualize the sequence of maps $c_j^S\circ \dots \circ c_1^S$  by starting with the identity permutation and underlining the letters to be rotated in the next step.
\begin{align*}
\id & = \underline{1\hspace{.2cm} 2\hspace{.2cm} 3\hspace{.2cm} 4\hspace{.2cm} 5\hspace{.2cm} 6}\hspace{.2cm} 7\\
c_1^S(\id) & = 6\hspace{.2cm} 1\hspace{.2cm} \underline{2\hspace{.2cm} 3\hspace{.2cm} 4\hspace{.2cm} 5}\hspace{.2cm} 7\\
c_2^S\circ c_1^S(\id) & = 6\hspace{.2cm} 1\hspace{.2cm} 5\hspace{.2cm} \underline{2\hspace{.2cm} 3}\hspace{.2cm} 4\hspace{.2cm} 7\\
\phi(S)=c_3^S\circ c_2^S\circ c_1^S(\id) & = 6\hspace{.2cm} 1\hspace{.2cm} 5\hspace{.2cm} 3\hspace{.2cm} 2\hspace{.2cm} 4\hspace{.2cm} 7\\
\end{align*}

Since we already know that $\left|{[n+j-1]\choose 2j}\right|=\left|\mathcal{E}^{(j)}_{n,231,j}\right|$, it suffices to show that $\phi$ is well-defined and injective. To show $\phi$ is well-defined, we induct on $j$. We want to show that
\begin{enumerate}
\item[(I)]$\phi(S)=\mathcal{E}^{(j)}_{n,231,j}$ whenever $S=\{a_1<a_2<\dots <a_{2j}\}\in {[n+j-1]\choose 2j}$,
\item[(II)]$\pi_{a_j}>\pi_{b_j}$, where $\phi(S)=\pi_1\dots \pi_n$, and $b_1,\dots ,b_j$ are as described in the definition of $\phi$,
\item[(III)]$i-\pi_i=j$ for $i=a_j+1,a_j+2,\dots ,b_j$.
\end{enumerate}

The base case is obvious (and coincides with the bijection given in Theorem \ref{thm:1} when $j=1$). Now let $j>1$ and assume the result holds for all $k<j$. Let $S=\{a_1<a_2<\dots <a_{2j}\}$, and let
\[b_j'=\begin{cases} b_j & \text{ if }b_j\in [n] \\ 
n+j-1 & \text{ otherwise} \end{cases}.\]
It follows that $\phi(S)=c_j^S(\phi(T))$, where $T=S-\{a_j,b_j'\}$. Let $\sigma=\phi(T)$, and let $\pi=\phi(S)$. Thus $\pi_i=\sigma_i$ if $i\notin [a_j,b_j]$, and $\pi_{a_j} \pi_{a_j+1} \dots \pi_{b_j} = \sigma_{b_j} \sigma_{a_j} \sigma_{a_j+1} \dots \sigma_{b_j-1}$. By induction we have that $i-\sigma_i=j-1$ for $i=a_{j-1}+1,a_{j-1}+2,\dots ,b_{j-1}$. Since $a_{j-1}<a_j<b_j\leq b_{j-1}$, then for $i=a_j+1,a_j+2,\dots ,b_j$ we have
\[i-\pi_i=i-\sigma_{i-1}=i-(i-j)=j,\]
which proves (III). Next we check the drop size of $\pi$ at position $a_j$
\[a_j-\pi_{a_j}=a_j-\sigma_{b_j}=a_j-(b_j-j+1)<b_j-(b_j-j+1)=j-1,\]
thus $\md[\pi]=j$. (II) follows immediately since
\[\pi_{a_j}=\sigma_{b_j}=b_j-j+1>b_j-j=\pi_{b_j}.\]
Next we prove the claim that $\des[\pi]=j$. By induction we know that $\des =j-1$ and $\sigma_{a_j}<\sigma_{a_j+1}<\dots <\sigma_{b_j}$. Moreover, (II) and (III) applied to $\pi$ shows us that $\Des(\pi)\cap [a_j,b_j-1]=\{a_j\}$. Since $\pi_{b_j}=b_j-j=\sigma_{b_j}-1$ and $\pi_{b_j+1}=\sigma_{b_j+1}$, it follows that $\pi_{b_j},\pi_{b_j+1}$ have the same relative order as $\sigma_{b_j},\sigma_{b_j+1}$. Thus $b_j\in \Des(\pi)$ if and only if $b_j\in \Des(\sigma)$. If $a_j-1\notin \Des(\sigma)$, then
\[\pi_{a_j-1}=\sigma_{a_j-1}<\sigma_{a_j}<\sigma_{b_j}=\pi_{a_j},\]
which implies $a_j-1\notin\Des(\pi)$. Conversely, suppose that $a_j-1\in \Des(\sigma)$. Since $\Des(\sigma)\cap [a_{j-1}+1,b_{j-1}]=\emptyset$, we have $a_j-1\leq a_{j-1}$. But since $a_{j-1}\leq a_j-1$, we have $a_j-1=a_{j-1}$. Then
\[\pi_{a_j-1}=\sigma_{a_j-1}=\sigma_{a_{j-1}}>\sigma_{b_{j-1}}\geq \sigma_{b_j}=\pi_{a_j},\]
which implies $a_j-1\in\Des(\pi)$. Since $\pi_i=\sigma_i$ for all $i\notin [a_j,b_j]$, it follows that $\Des(\pi)=\Des(\sigma)\biguplus\{a_j\}$, thus $\des[\pi]=\des +1=j$ which proves (I).

Next we prove that $\phi$ is injective. Let $S=\{a_1<a_2<\dots <a_{2j}\}\in {[n+j-1]\choose 2j}$ from which we construct the sequence of intervals $I_1^S\supset I_2^S \supset \dots \supset I_j^S$ with $I_m^S=[a_m,b_m]$ for $m=1,\dots j$.. And let $U=\{x_1<x_2<\dots <x_{2j}\}\in {[n+j-1]\choose 2j}$ from which we construct the sequence of intervals $I_1^U\supset I_2^U \supset \dots \supset I_j^U$ with $I_m^U=[x_m,y_m]$ for $m=1,\dots, j$. Suppose $S\neq U$, then we claim there is at least one index $m$ such that $I_m^S\neq I_m^U$. Suppose there is some $a_k\in [n]\cap S$ such that $a_k\notin U$. Then there is at least one interval $I_m^S$ such that $a_m=a_k$ or $b_m=a_k$, but there is no interval $I_m^U$ with $x_m=a_k$ or $y_m=b_k$. Next suppose that for some $m$ with $2\leq m \leq j$ we have $n+m-1\in S$, but $n+m-1\notin U$. Then $b_m=b_{m-1}$ and $y_m\neq y_{m-1}$. It follows that either $I_m^S\neq I_m^U$ or $I_{m-1}^S\neq I_{m-1}^U$. This proves the claim.

Let $\pi=\phi(S)$ and $\sigma=\phi(U)$. It follows from (II) and (III) above and the fact that $a_1<a_2<\dots <a_j$, that $\Des(\pi)=\{a_1,a_2,\dots ,a_j\}$. Similarly, $\Des(\sigma)=\{x_1,x_2,\dots ,x_j\}$. If $a_m\neq x_m$ for some $m=1,2,\dots ,j$, then $\Des(\pi)\neq \Des(\sigma)$ and $\pi\neq\sigma$. It follows from (III) above and from the definition of $\phi$ that $i-\pi_i=j-k$ for $i=b_{j-k+1}+1,b_{j-k+1}+2,\dots ,b_{j-k}$ and $k=1,2,\dots j-1$. It is also clear that $\pi_i=i$ for $i>b_1$. Now suppose that $a_m\neq x_m$ for some $m=j+1,j+2,\dots ,2j$. This implies that $b_m\neq y_m$, so assume $b_m>y_m$. First we note that $b_m-\pi_{b_m}\geq m$, since $\phi$ reduces the letter in position $b_m$ by one at least $m$ times (it will be reduced by one more than $m$ times if $b_{m-1}=b_m$). On the other hand since $y_m<b_m$ we have $b_m-\sigma_{b_m}<m$, thus $\pi\neq\sigma$. This proves $\phi$ is injective.

\end{proof}

A method similar to the one used to prove Theorem \ref{exact k=j}, proves the following theorem.

\begin{theorem}\label{exact k=j-1}

For all $j\geq 2$ we have
\[e_{n,231,j}^{(j-1)}=(2j-3){n+j-2\choose 2j}.\]

Consequently,
\[a_{n,231,j}^{(j-2)}=\frac{1}{n}{n\choose j}{n\choose j+1}-{n+j-1\choose 2j}-(2j-3){n+j-2\choose 2j}.\]

\end{theorem}

\begin{proof}

We interpret $e_{n,231,j}^{(j-1)}$ as the number of ordered trees with $n$ edges, $j+1$ internal nodes, and height equal to $j$. Such trees have the form  shown in Figure \ref{fig:exact2} 
\begin{figure}[h]
\centering
\includegraphics[height=7.5cm]{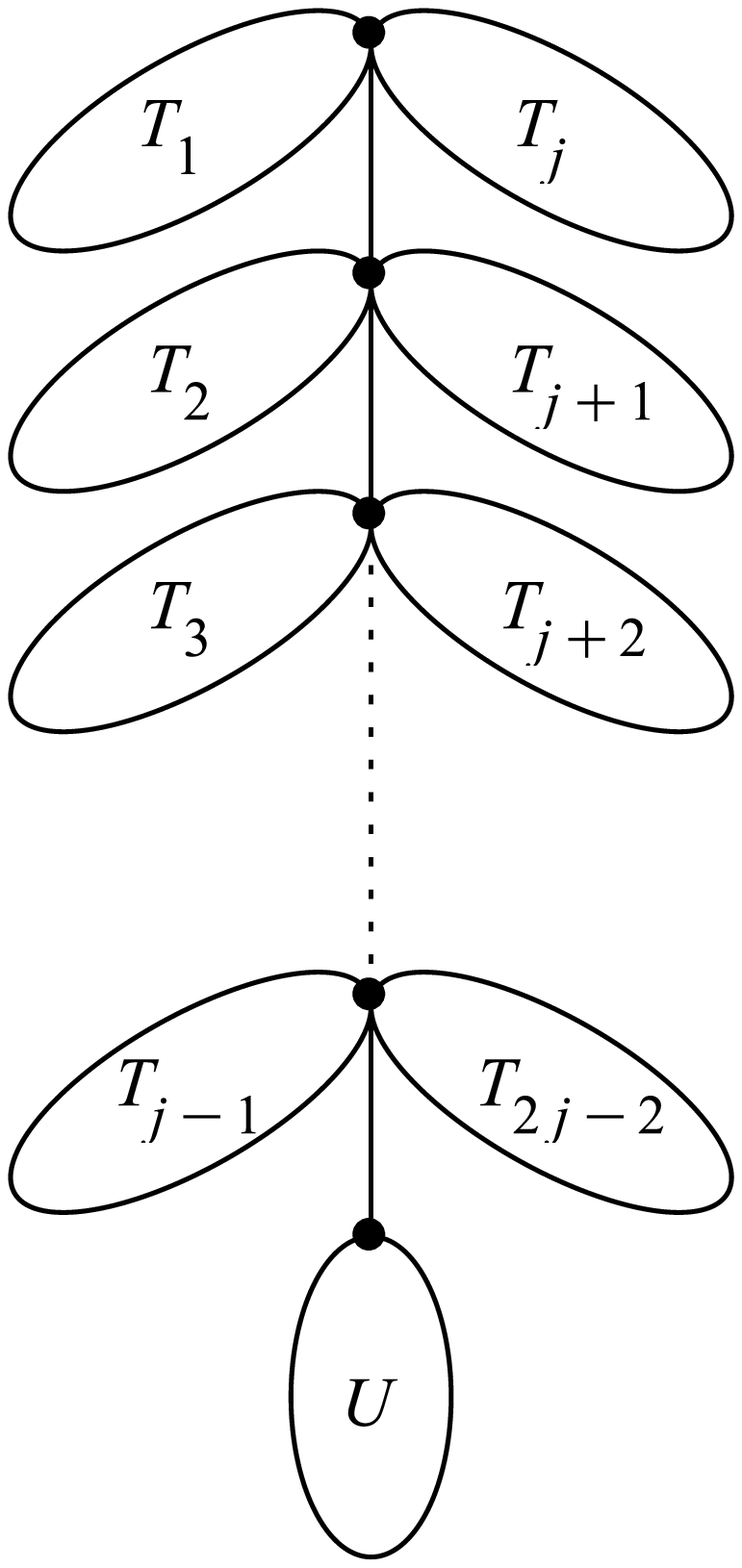}
\caption{An ordered tree with $j+1$ internal nodes and height equal to $j$.}
\label{fig:exact2}
\end{figure}
where the subtree $U\in\mathcal{T}_{p,1}^{(1)}$ with $p\geq 1$, and the subtree $T_{j-1}\in\mathcal{T}_{m_{j-1},1}^{(1)}$ with $m_{j-1}\geq 0$. The remaining subtrees $T_r\in\mathcal{T}_{m_r,d_r}^{(2)}$ with $m_r\geq 0$ for $1\leq r\leq 2j-2$ with $r\neq j-1$. The idea here is that the leftmost leaf in $U$ is the first vertex (in postorder) to reach the full height $j$. We also require that exactly one of the subtrees $T_r$ with $r\neq j-1$ has two internal nodes so that the resulting tree has $j+1$ total internal nodes, i.e.
\[\sum_{\substack{1\leq r\leq 2j-2 \\ r\neq j-1}}d_r=2j-2.\]
To get a total of $n$ edges, we also require that
\[p+\sum_{r=1}^{2j-2}m_r=n-j+1.\]
Then every such choice of subtrees $U,T_1,T_2,\dots ,T_{2j-2}$ corresponds to a unique tree with $n$ edges, $j+1$ internal nodes, and height equal to $j$. Since $\left|\mathcal{T}_{m_r,d_r}^{(2)}\right|=a_{m_r,231,d_r-1}^{(1)}={m_r\choose 2d_r-2}$. Thus
\begin{align*}
e_{n,231,j}^{(j-1)}&=\sum_{\substack{(p;m_1,\dots ,m_{2j-2})\in W_{n-j+1}(1,2j-2) \\
(d_1,d_2,\dots ,d_{j-2},d_j,d_{j+1},\dots ,d_{2j-2})\in W_{2j-2}(2j-3,0)}}\prod_{r\neq j-1}{m_r\choose 2d_r-2}\\
&=(2j-3)\left(\sum_{(p;m_1,\dots ,m_{2j-2})\in W_{n-j+1}(1,2j-2)}{m_1\choose 2}\right)\\
&=(2j-3)\left(\sum_{(p,m_1,\dots ,m_{2j-2})\in W_{n-j}(0,2j-1)}{m_1\choose 2}\right)\\
&=(2j-3)\sum_{m_1=0}^{n-j}{m_1\choose 2}\sum_{(p;m_2,\dots ,m_{2j-2})\in W_{n-j-m_1}(0,2j-2)}1\\
&=(2j-3)\sum_{m_1=0}^{n-j}{m_1\choose 2}|W_{n-j-m_1}(0,2j-2)|\\
&=(2j-3)\sum_{m_1=0}^{n-j}{m_1\choose 2}{n+j-m_1-3\choose 2j-3}\\
&=(2j-3)\sum_{m_1=0}^{n-j}{m_1\choose 2}{n+j-3-m_1\choose (2j-1)-2}\\
&=(2j-3){n+j-2\choose 2j}.
\end{align*}
The last step follows from the identity
\[\sum_{m=0}^{n}{m\choose j}{n-m\choose k-j}={n+1\choose k+1},\]
which holds for all $n\geq k\geq j\geq 0$, and can easily be proved by induction.

\end{proof}

A similar method could be used to compute $e_{n,231,j}^{(j-2)}$. However, the proof continues to become more complicated. We hope the reader is convinced that the proof and resulting formula for $e_{n,231,j}^{(j-2)}$ will be somewhat unpleasant, and that this method will become even more unpleasant as we continue to lower the maximum drop size. Instead, we will show later that $a_{n,j}^{(k)}$ can be expressed as a (positive) sum of products of binomial coefficients.

Our next goal is to obtain a recurrence for $a_{n,231,j}^{(k)}$. We accomplish this using a bijection to find a recurrence for trees with $n$ edges, $j$ leaves, and height less than or equal to $k$. Let $\mathcal{N}(n,j,k)$ denote the set of ordered trees with $n$ edges, $j$ leaves, and height less than or equal to $k$, and let $N(n,j,k)=\left|\mathcal{N}(n,j,k)\right|$. In other words, $N(n,j,k)$ are the Narayana numbers refined by height. For convenience, we let $\mathcal{N}(0,0,k)$ be the set containing the tree with one vertex and no edges, hence $N(0,0,k)=1$ for all $k\geq 0$. Note that in terms of Dyck paths we have $N(n,j,k)=\left|\mathcal{D}_{2n,j}^{(k)}\right|$, i.e. the number of Dyck paths of length $2n$ with $j$ peaks and height less than or equal to $k$. We find a recurrence for $N(n,j,k)$, which can easily be translated into a recurrence for $a_{n,231,j}^{(k)}$.

\begin{theorem}\label{rec trees}
For all $k\geq 1$ and for all $n\geq j\geq 1$ we have
\begin{equation}\label{general N rec1}
N(n,j,k)=\sum_{i=0}^{n-j}N(n-j,i,k-1){2n-j-i\choose 2n-2j}.
\end{equation}

By replacing $i$ with $n-j-i$ we obtain
\begin{equation}\label{general N rec2}
N(n,j,k)=\sum_{i=0}^{n-j}N(n-j,n-j-i,k-1){n+i\choose 2n-2j}.
\end{equation}

\end{theorem}




\begin{proof}


We construct a map 
\[s:\bigcup_{i=0}^{n-j}\mathcal{N}(n-j,i,k-1)\times W_j(i,2n-2j-i+1)\rightarrow \mathcal{N}(n,j,k),\]
which we will show is a bijection. Let $T\in \mathcal{N}(n-j,i,k-1)$, and let 
\[c=(l_1,\dots ,l_i;n_1,\dots ,n_{2n-2j-i+1})\in W_j(i,2n-2j-i+1)\]
for some $i$ such that $0\leq i \leq n-j$. We describe $s(T,c)$ via a composition of maps, $s=s_1\circ s_2 \circ \dots \circ s_k$. Let $U_h=s_h\circ s_{h+1}\circ \dots \circ s_k(T,c)$ with $2\leq h\leq k$, and set $U_{k+1}=(T,c)$. Construct an ordered tree $U_{h-1}=s_{h-1}(U_h)$ by visiting the vertices at level $h-1$ of $U_h$ from right to left, and (possibly) adding edges to each vertex $x$ as follows:
\begin{itemize}
\item If $x$ is the $p^{\text{th}}$ leaf visited in the process of applying $s_{h-1},s_h,\dots ,s_k$ to $(T,c)$, then attach the unique tree from $\mathcal{N}(l_p,l_p,1)$ as a subtree below $x$.
\item If $x$ is an internal node with degree say $d$ (i.e. $x$ has $d$ children), then attach the tree from $\mathcal{N}(n_r,n_r,1)$ to the right of the rightmost edge below $x$. Here $r=1+\sum (\deg(y)+1)$ where the sum is over all internal nodes previously visited in the process of applying  $s_{h-1},s_h,\dots ,s_k$ to $(T,c)$. Then for $m=1,2,\dots ,d$, attach the tree from $\mathcal{N}(n_{r+m},n_{r+m},1)$ to the left of the $m$-th rightmost edge below $x$.
\end{itemize}

Consider the following example. Let $n=14, j=10, k=3, i=3$. Let $T\in\mathcal{N}(4,3,2)$ be the tree shown in Figure \ref{fig:rec0}.
\begin{figure}[H]
\centering
\includegraphics[height=7cm]{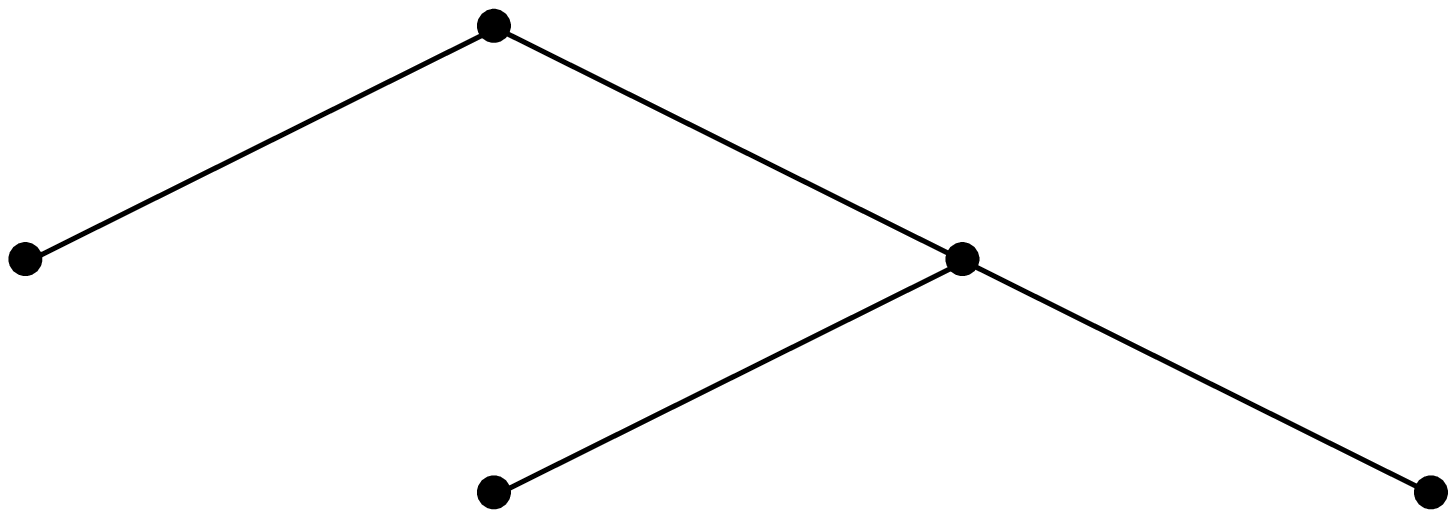}
\caption{An ordered tree $T\in\mathcal{N}(4,3,2)$.}
\label{fig:rec0}
\end{figure}
Let
\[c=(1,2,2;1,2,0,0,1,1)\in W_{10}(3,6).\]
Find $s_3(T,c)$ by visiting the vertices at level 2 from right to left. Both vertices are leaves, so we attach $l_1=1$ edge to the right vertex, and $l_2=2$ edges to the left vertex, where the dashed edges represent the added edges (see Figure \ref{fig:rec1}).
\begin{figure}[H]
\centering
\includegraphics[height=7cm]{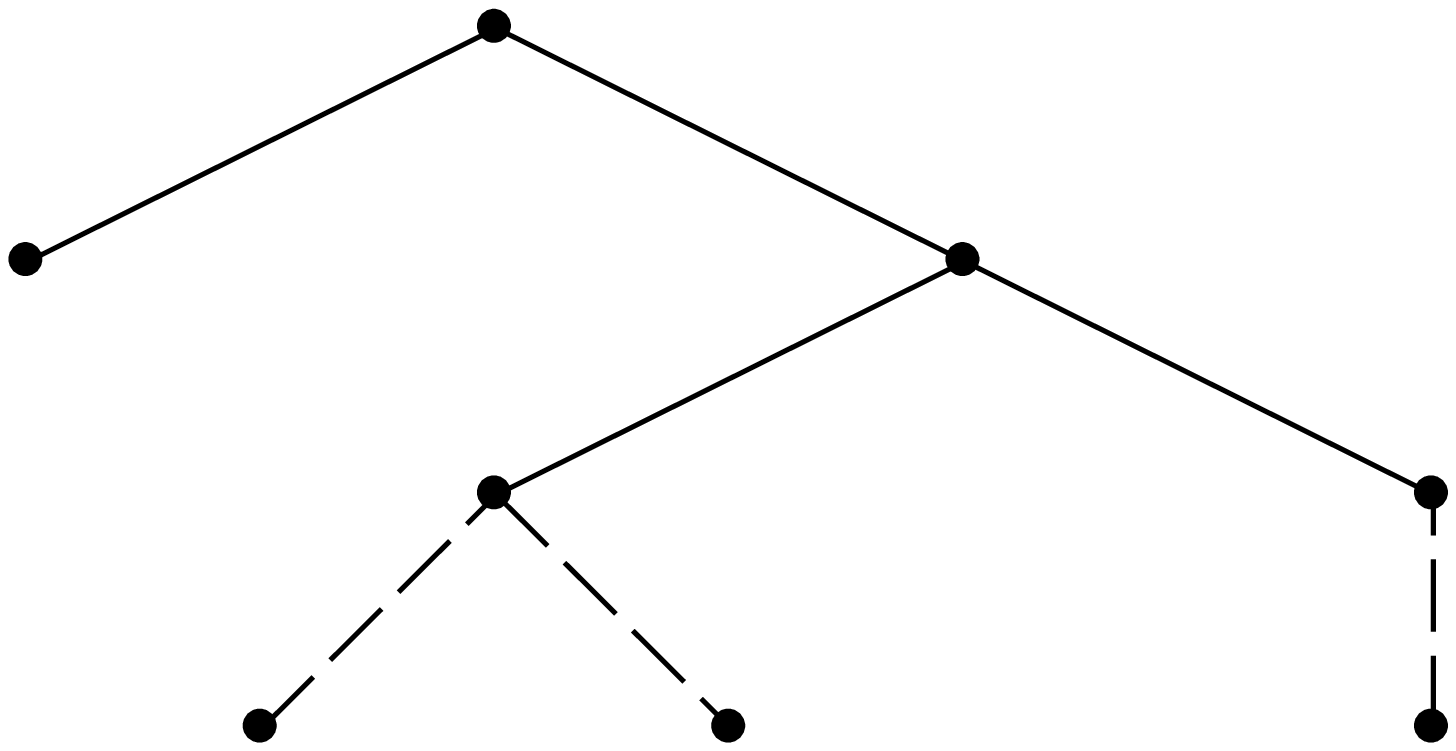}
\caption{$s_3(T,c), \hspace{.2cm}c=(\mathbf{1},\mathbf{2},2;1,2,0,0,1,1)$.}
\label{fig:rec1}
\end{figure}
We continue by applying $s_2$. The rightmost vertex on level 1 is an internal node, so we attach $n_1=1$ edge to the right of the rightmost edge, attach $n_2=2$ edges in the middle, and attach $n_3=0$ edges to the left of the leftmost edge. The next vertex on level 1 is a leaf, so we attach $l_3=2$ edges below this vertex (see Figure \ref{fig:rec2}).
\begin{figure}[H]
\centering
\includegraphics[height=7cm]{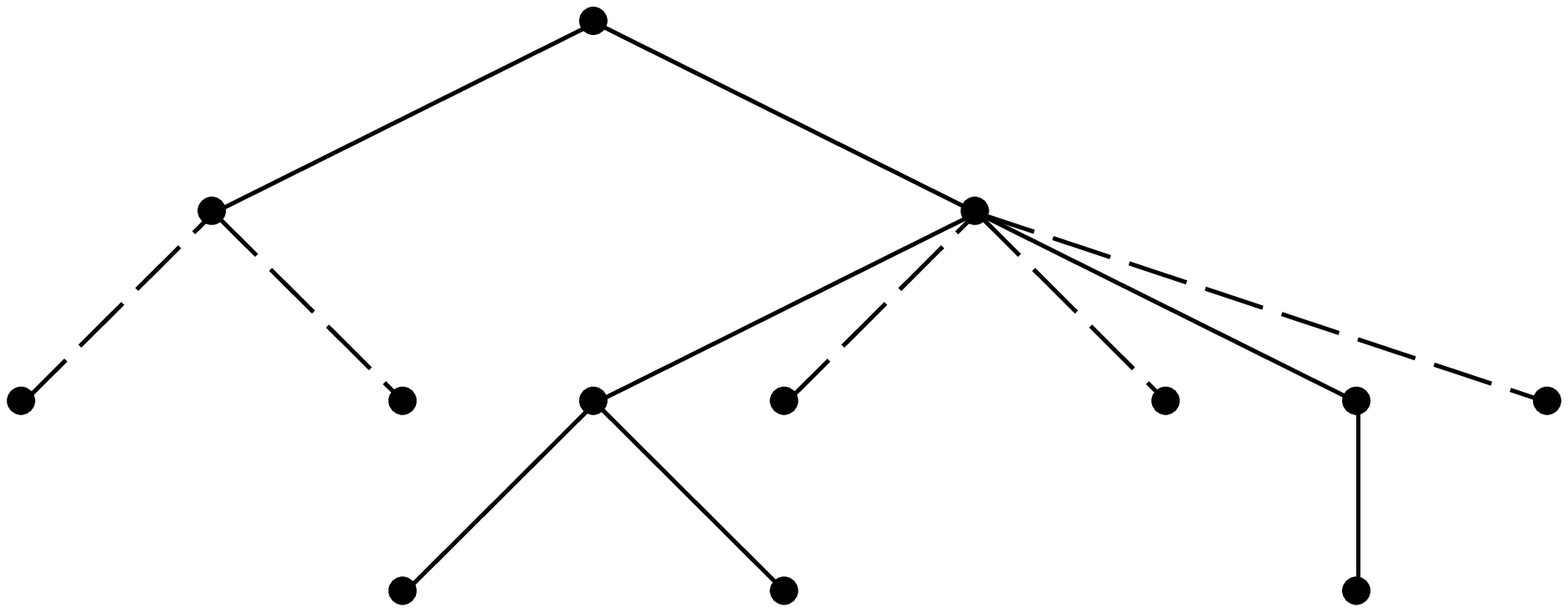}
\caption{$s_2\circ s_3(T,c), \hspace{.2cm}c=(1,2,\mathbf{2};\mathbf{1},\mathbf{2},\mathbf{0},0,1,1)$.}
\label{fig:rec2}
\end{figure}
Lastly we apply $s_1$ to obtain $s(T,c)=s_1\circ s_2\circ s_3 (T,c)$ (see Figure \ref{fig:rec3}).
\begin{figure}[H]
\centering
\includegraphics[height=7cm]{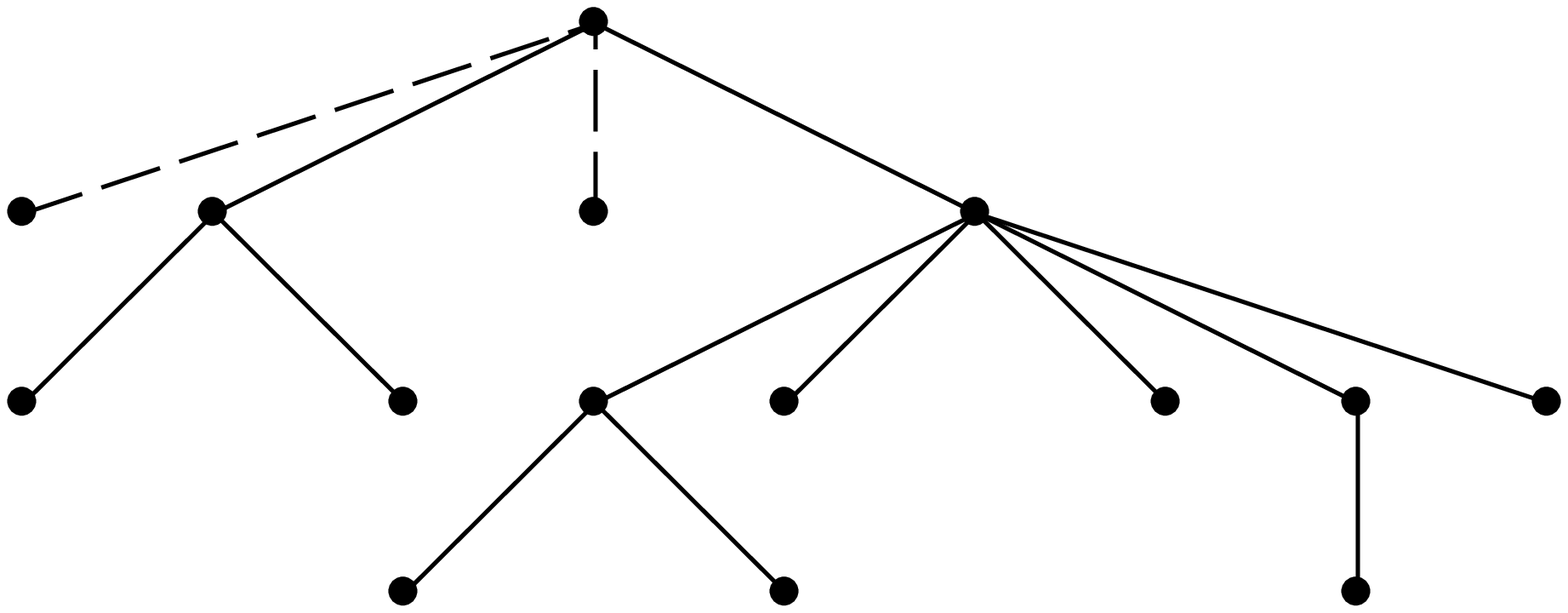}
\caption{$s_1\circ s_2\circ s_3(T,c), \hspace{.2cm}c=(1,2,2;1,2,0,\mathbf{0},\mathbf{1},\mathbf{1})$.}
\label{fig:rec3}
\end{figure}

Next we show that $s$ is well-defined. Note that $T$ has $n-j$ edges and $c$ is a weak composition of $j$. Thus applying $s$ will add $j$ edges to $T$, so $s(T,c)$ has $n$ edges. Since $l_1,\dots ,l_i$ are all positive, every leaf of $T$ has edges added to it, and is therefore not a leaf in $s(T,c)$. On the other hand, every edge added to $T$ creates a leaf, so $s(T,c)$ has $j$ leaves. Since $T$ has height less than or equal to $k-1$, it is clear that $s(T,c)$ has height less than or equal to $k$. Furthermore, $T$ has $i$ leaves and the first $i$ parts of $c$ are positive. We also need to check that $c$ has the appropriate number of parts for adding edges to internal nodes. This follows from the fact that
\begin{align*}
\sum_{\substack{x\text{ is an internal}\\ \text{node of }T}}(1+\deg(x))&
=|\{\text{internal nodes of }T\}|+\sum_{\substack{x\text{ is an internal}\\ \text{node of }T}}\deg(x)\\
&=|\{\text{vertices of }T\}|-|\{\text{leaves of }T\}|+|\{\text{edges of }T\}|\\
&=(n-j+1)-i+(n-j)\\
&=2n-2j-i+1
\end{align*}

Next we describe the inverse map of $s$, which we denote by $f$. We have chosen the letter $s$ to correspond to spring, since the tree "grows" edges during this map. And the letter $f$ corresponds to fall since we will remove edges during this map. Let $T\in\mathcal{N}(n,j,k)$. Again we describe $f(T)$ via a composition of maps $f_k\circ f_{k-1}\circ \dots \circ f_1$. Let $(V_h,c_h)=f_h\circ f_{h-1}\circ \dots \circ f_1(T)$ with $1\leq h\leq k-2$, where $V_h$ is an ordered tree and $c_h$ is a weak composition, and let $(V_0,c_0)=(T,\emptyset)$. Construct $(V_{h+1},c_{h+1})=f_{h+1}(V_h,c_h)$ by visiting the vertices at level $h$ of $V_h$ from \textit{left to right}, and removing all single edges below each vertex. The weak composition $c_{h+1}$ is obtained from $c_h$ by recording at each vertex $x$, the numbers of edges removed as follows:
\begin{itemize}
\item If $x$ has $p$ children and all subtrees below $x$ have height one (in other words $x$ has no grandchildren), then append a $p$ to the beginning of the positive parts of $c_h$.


\item Suppose $x$ has $Y_1,Y_2,\dots ,Y_d$ (from left to right) subtrees of height greater than or equal to 2. Let $r_1$ be the number of single edges below $x$ and to the left of $Y_1$. For $m=1,2,\dots ,d-1$, let $r_m$ equal the number of single edges below $x$ between $Y_{m}$ and $Y_{m+1}$. Let $r_{d+1}$ be the number of single edges below $x$ and to the right of $Y_{d}$. Then append the parts $(r_{d+1},r_{d},\dots,r_1)$ to the beginning of the nonnegative parts of $c_h$.
\end{itemize}

Here is an example, let $T=V_0\in\mathcal{N}(10,7,3)$ be the ordered tree in Figure \ref{fig:f0}.
\begin{figure}[H]
\centering
\includegraphics[height=7cm]{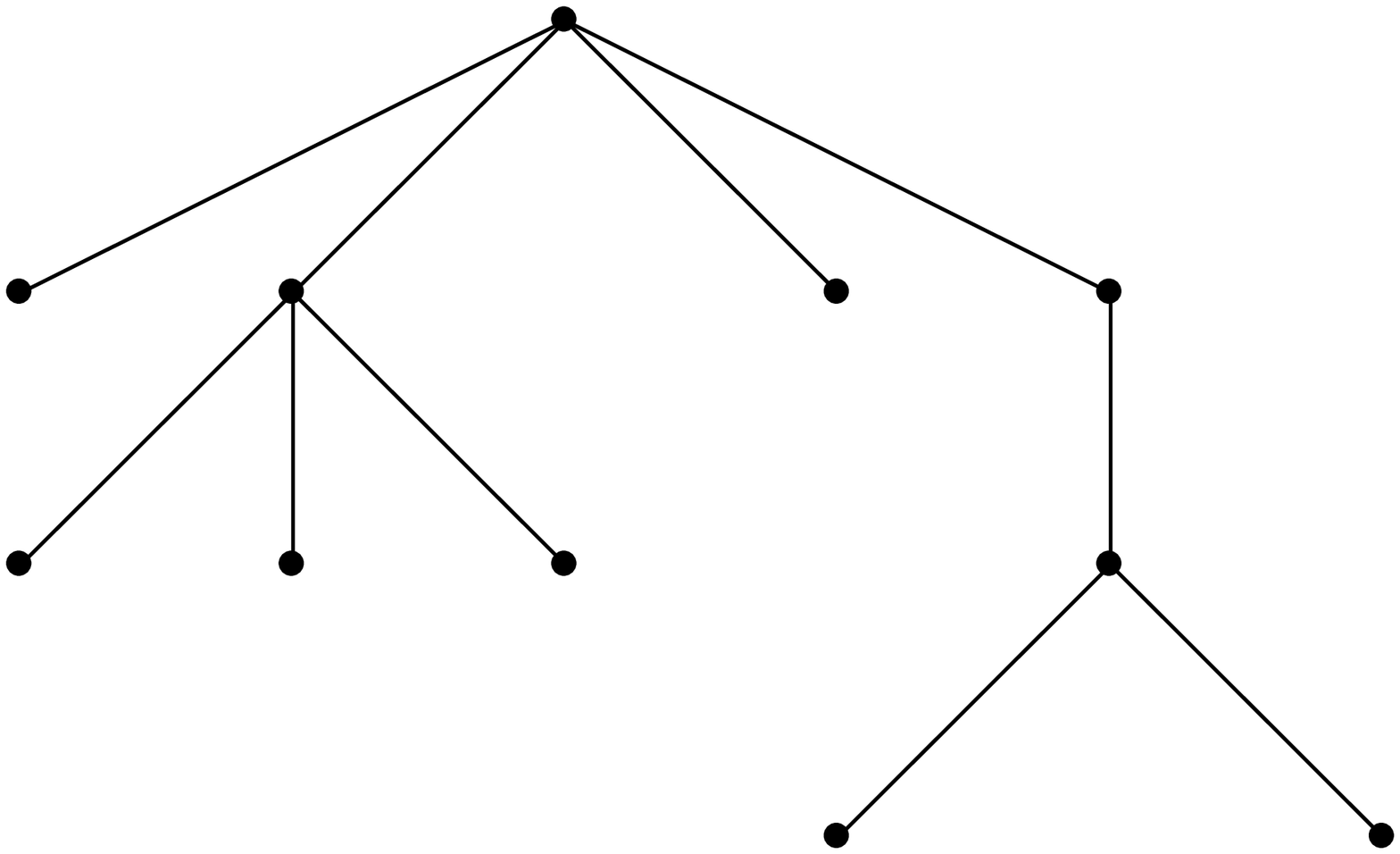}
\caption{An ordered tree $T=V_0$, and set $c_0=\emptyset$.}
\label{fig:f0}
\end{figure}
The root has 2 subtrees $Y_1$ and $Y_2$ with height greater than or equal to 2. There is $n_1=1$ single edge to the left of $Y_1$, there is $n_2=1$ single edge to the right of $Y_1$, and $n_3=0$ single edges to the right of $Y_2$. We remove these single edges and record the number of edges removed as nonnegative parts of the weak composition $c_1$, i.e. $c_1=(n_3,n_2,n_1)=(0,1,1)$ (see Figure \ref{fig:f1}).
\begin{figure}[H]
\centering
\includegraphics[height=7cm]{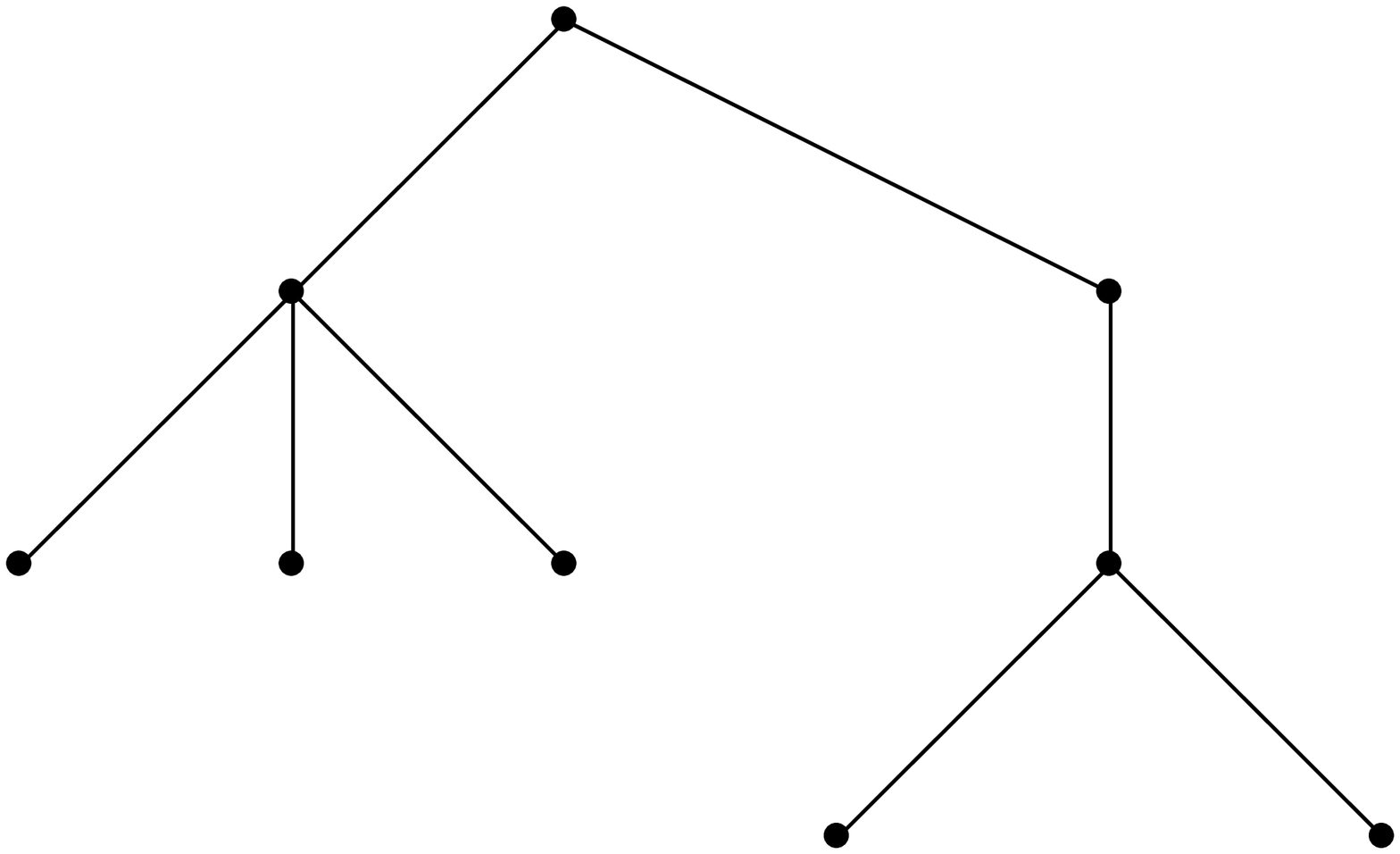}
\caption{The ordered tree $V_1$, and $c_1=(0,1,1)$.}
\label{fig:f1}
\end{figure}
Next we visit the vertices at level 1. The first vertex (moving from left to right) has only single edges. We remove these $l_1=3$ edges and record the number of edges removed as a positive part of the weak composition $c_2$. The next vertex has one subtree of height 2. There are no single edges, so we record two zeros as nonnegative parts of the weak composition $c_2$, i.e. $n_4=n_5=0$ (see Figure \ref{fig:f2}).
\begin{figure}[H]
\centering
\includegraphics[height=7cm]{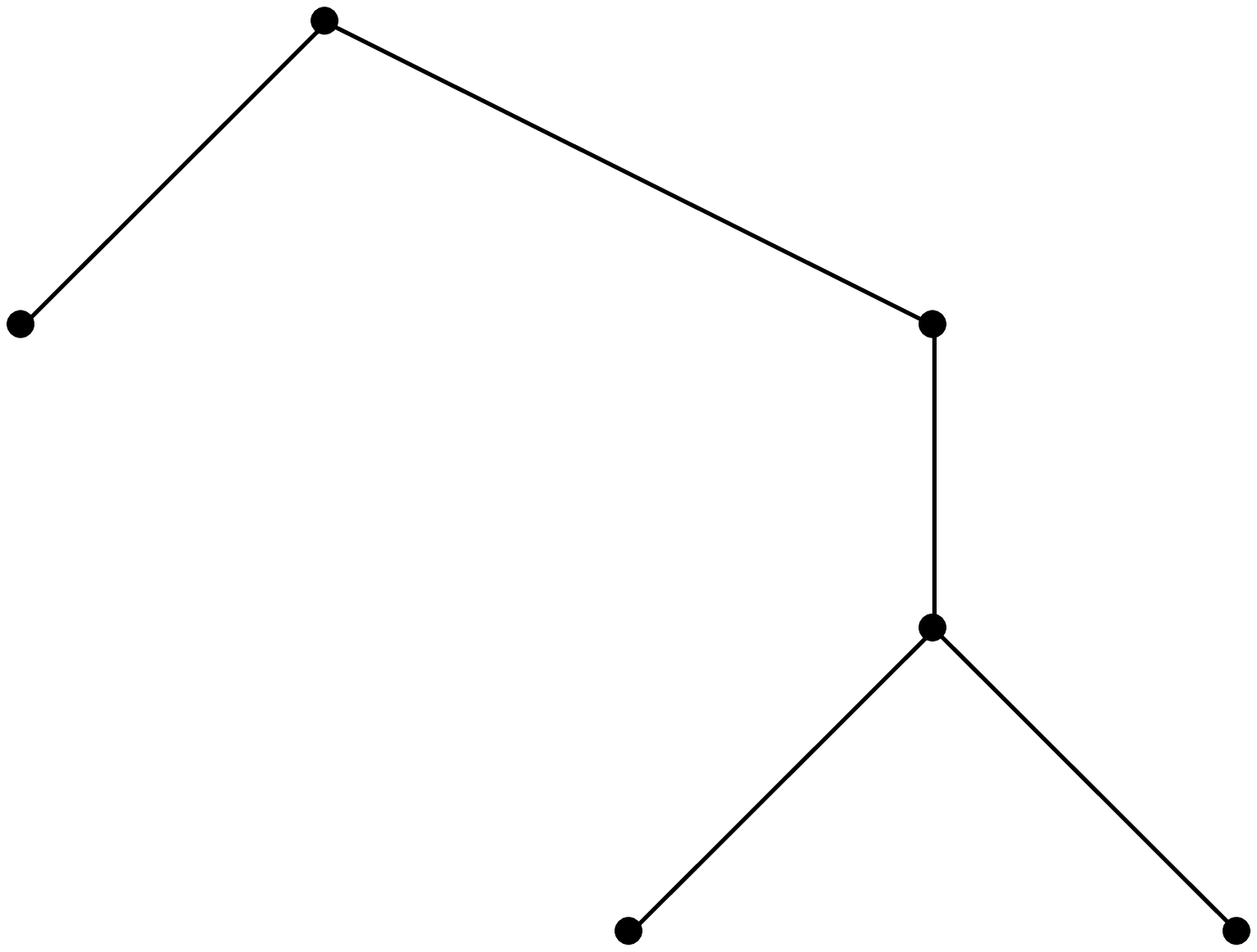}
\caption{The ordered tree $V_2$, and $c_2=(3;0,0,0,1,1)$.}
\label{fig:f2}
\end{figure}
At level 2, there is one vertex with only single edges, so we remove them and record the number of edges removed as a positive part of the weak composition $c_3$, i.e. $l_2=2$ (see Figure \ref{fig:f3}).
\begin{figure}[H]
\centering
\includegraphics[height=7cm]{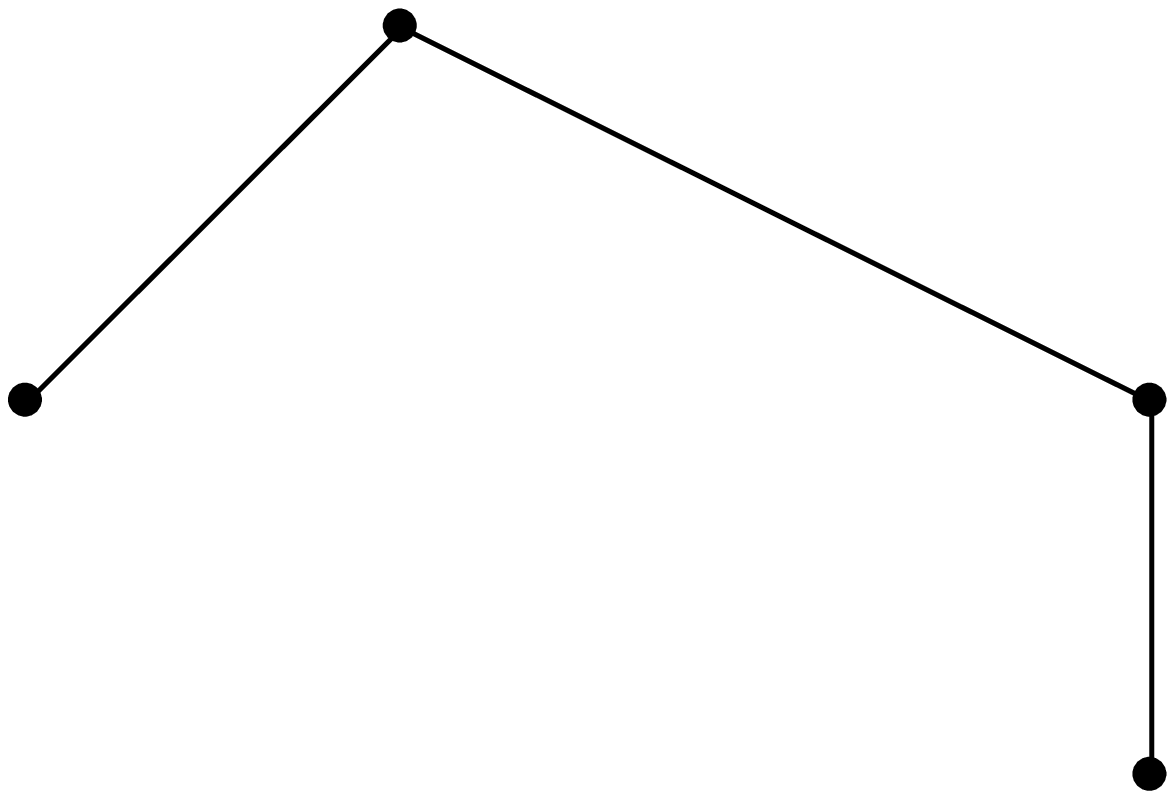}
\caption{The ordered tree $V_3$, and $c_3=(2,3;0,0,0,1,1)$.}
\label{fig:f3}
\end{figure}

Next we show $f$ is well-defined. When applying $f_h$, we never visit a vertex which is a leaf since such a vertex would have been removed when applying $f_{h-1}$. Since we remove from $T$ precisely all edges which have a leaf at the bottom, we see that $V_k$ has $n-j$ edges. Since $V_k$ has $n-j$ edges, the number of leaves of $V_k$ is less than or equal to $n-j$. Clearly, the height of $V_k$ is one less than the height of $T$. Thus $V_k\in\mathcal{N}(n-j,i,k-1)$ for some $0\leq i \leq n-j$.

Since the map $f$ removes $j$ edges from $T$, and since $c_k$ records the total number of edges removed, $c_k$ is a weak composition of $j$. A leaf of $V_k$ is created only when we visit a vertex with only single edges below. The number of such edges is recorded as a positive part in the weak composition $c_k$. So if $V_k$ has $i$ leaves, then $c_k$ has $i$ positive parts. Lastly, the total number of parts (positive and nonnegative) of $c_k$ is given by
\begin{align*}
\sum_{\substack{x\text{ is a}\\ \text{vertex of }V_k}}(1+\deg(x))&
=|\{\text{vertices of }V_k\}|+|\{\text{edges of }V_k\}|\\
&=(n-j+1)+(n-j)\\
&=2n-2j+1,
\end{align*}
thus $c_k$ has $2n-2j+1-i$ nonnegative parts.

It is clear by construction that $f$ is the inverse of $s$.

The Theorem now follows from the fact that (see Proposition \ref{weak comp size})
\[\left|W_j(i,2n-2j-i+1)\right|={2n-j-i\choose 2n-2j}.\]

\end{proof}

The following is a recurrence for the Narayana number due to Zabrocki (see A001263 in the OEIS \cite{oeis}).
\begin{corollary}\label{Zabrocki}
For $n\geq j\geq 2$ we have
\[N(n,j)=\sum_{i=1}^{j-1}N(j-1,i){n+i-1\choose 2j-2}.\]
\end{corollary}

\begin{proof}
Assume $j\geq 2$. If $k\geq n$, then $N(n,j,k)=N(n,j)$ and $N(n-j,n-j-i,k-1)=N(n-j,n-j-i)$ so that \eqref{general N rec2} becomes
\[N(n,j)=\sum_{i=0}^{n-j}N(n-j,n-j-i){n+i\choose 2n-2j}.\]
Since the Narayana number are symmetric, $N(n,j)=N(n,n-j+1)$, we have
\begin{align*}
N(n,j)&=N(n,n-j+1)\\
&=\sum_{i=0}^{j-1}N(j-1,j-1-i){n+i\choose 2j-2}\\
&=\sum_{i=0}^{j-2}N(j-1,i+1){n+i\choose 2j-2}\\
\end{align*}
The result now follows from replacing $i$ by $i-1$.

\end{proof}

Theorem \ref{rec trees} also gives us a recurrence for $a_{n,231,j}^{(k)}$.

\begin{corollary}\label{anjk rec}
For $k\geq 1$ we have
\[a_{n,231,j}^{(k)}=\sum_{i=0}^{j}a_{j,231,i}^{(k-1)}{n+i \choose 2j},\]
where $a_{0,231,0}^{(k-1)}:=1$ for all $k\geq 1$.
\end{corollary}

\begin{proof}
\[a_{n,231,j}^{(k)}=\left|\mathcal{T}_{n,j+1}^{(k+1)}\right|
=N(n,n-j,k+1)
=\sum_{i=0}^{j}N(j,j-i,k){n+i\choose 2j}\]
\[=\sum_{i=0}^{j}a_{j,231,i}^{(k-1)}{n+i\choose 2j}.\]

\end{proof}

The recurrence in Corollary \ref{anjk rec} can be iterated to obtain a closed form expressions for $a_{n,231,j}^{(k)}$. Indeed, $a_{j,231,i}^{(0)}$ is the number of permutations in $S_j(231)$ with $i$ descents and no drops. Since the identity permutation is the only permutation with no drops, we see that $a_{j,231,i}^{(0)}=1$ if $i=0$, and $a_{j,231,i}^{(0)}$ is zero otherwise. Hence
\[a_{n,231,j}^{(1)}=\sum_{i=0}^{j}a_{j,231,i}^{(0)}{n+i\choose 2j}={n \choose 2j},\]
as expected (see Theorem \ref{thm:1}). We iterate to obtain the following formula.

\begin{corollary}\label{anj2 closed}
For all $n,j\geq 0$ we have

\[a_{n,231,j}^{(2)}=\sum_{j\geq i\geq 0}a_{j,231,i}^{(1)}{n+i \choose 2j}
=\sum_{j\geq i\geq 0}{j\choose 2i}{n+i \choose 2j}.\]

\end{corollary}

We continue iterating to obtains more formulas. In each case the formula holds for all $n,j\geq 0$.


\[a_{n,231,j}^{(3)}
=\sum_{j\geq i_2\geq i_1\geq 0}{i_2\choose 2i_1}{j+i_1\choose 2i_2}{n+i_2\choose 2j}.\]


\[a_{n,231,j}^{(4)}
=\sum_{j\geq i_3\geq i_2\geq i_1\geq 0}{i_2\choose 2i_1}{i_3+i_1\choose 2i_2}{j+i_2\choose 2i_3}{n+i_3\choose 2j}.\]

\[a_{n,231,j}^{(5)}
=\sum_{j\geq i_4\geq i_3\geq i_2\geq i_1\geq 0}{i_2\choose 2i_1}{i_3+i_1\choose 2i_2}{i_4+i_2\choose 2i_3}{j+i_3\choose 2i_4}{n+i_4\choose 2j}.\]

A pattern emerges, giving us a formula for $a_{n,231,j}^{(k)}$.

\begin{theorem}\label{anjk closed}
For all $n,j\geq 0$ and all $k\geq 2$ we have
\[a_{n,231,j}^{(k)}=\sum_{j\geq i_{k-1}\geq \dots \geq i_1\geq 0}
\left(\prod_{m=0}^{k-1}{i_{m+2}+i_m\choose 2i_{m+1}}\right),\]
where $i_0:=0,\hspace{.2cm} i_k:=j,\hspace{.2cm} i_{k+1}:=n$.

\end{theorem}

\begin{proof}

Induct on $k$. The base case $k=2$ follows from Corollary \ref{anj2 closed}. Now let $k\geq 3$ and assume the result holds for $k-1$. Then from Corollary \ref{anjk rec} we have
\[a_{n,231,j}^{(k)}=\sum_{j\geq p_{k-1}\geq 0}a_{j,231,p_{k-1}}^{(k-1)}{n+p_{k-1} \choose 2j}.\]
Use the induction hypothesis to substitute an expression for $a_{j,231,p_{k-1}}^{(k-1)}$.
\[a_{n,231,j}^{(k)}=\sum_{\substack{j\geq p_{k-1}\geq 0\\ p_{k-1}\geq i_{k-2}\geq \dots \geq i_1\geq 0}}
\left(\prod_{m=0}^{k-4}{i_{m+2}+i_m\choose 2i_{m+1}}\right){p_{k-1}+i_{k-3}\choose 2i_{k-2}}
{j+i_{k-2}\choose 2p_{k-1}}{n+p_{k-1} \choose 2j}.\]
The result now follows from replacing $p_{k-1}$ with $i_{k-1}$.

\end{proof}

We translate Theorem \ref{anjk closed} into a an explicit formula for $N(n,j,k)$ (which may be interpreted in terms of ordered trees, or in terms of Dyck paths).


\begin{corollary}\label{Nnjk closed}
For $n\geq j\geq 0$ and $k\geq 3$ we have
\[N(n,j,k)=a_{n,231,n-j}^{(k-1)}=\sum_{n-j\geq i_{k-2}\geq \dots \geq i_1\geq 0}
\left(\prod_{m=0}^{k-2}{i_{m+2}+i_m\choose 2i_{m+1}}\right),\]
where $i_0:=0,\hspace{.2cm} i_{k-1}:=n-j, \hspace{.2cm} i_{k}:=n$.

\end{corollary}

\section{Resulting Identities}\label{identities}

In the previous section we proved that the set of permutations in $S_n(231)$ with $j$ descents and maximum drop less than or equal to $k$, is in bijective correspondence with the set of ordered trees with $n$ edges, $j+1$ internal nodes, and height less than or equal to $k+1$. We also found two seemingly different closed form expressions for the number of such trees: one due to Kemp \cite{kemp} (Theorem \ref{closed general trees} and Corollary \ref{closed general perms}), and another resulting from iterating our recurrence (Theorem \ref{anjk closed}). This leads to some remarkable identities.

\begin{theorem}\label{identity anj1}
For $n\geq 1$ and $j\geq 0$ we have
\begin{equation}\label{identity anj1 eq1}
a_{n,231,j}^{(1)}=h_3(n+1,n-j).
\end{equation}

Consequently
\begin{equation}\label{identity anj1 eq2}
{n \choose 2j}=N(n,j+1)-\left[\widetilde{Q}_1(n,j,3)-2\widetilde{Q}_0(n,j,3)+\widetilde{Q}_{-1}(n,j,3)\right],
\end{equation}
where
\[N(n,j+1)=\frac{1}{n}{n\choose j+1}{n\choose j},\]
and
\[\widetilde{Q}_a(n,j,3)=\sum_{s\geq 1}{n-2s-1\choose j-3s-a}{n+2s-1\choose j+3s+a}.\]

\end{theorem}

\begin{proof}
First note that \eqref{identity anj1 eq1} is just a special case of Corollary \ref{closed general perms} with $k=1$.

The left hand side of \eqref{identity anj1 eq2} follows from Theorem \ref{thm:1}. While the right hand side of \eqref{identity anj1 eq2} follows from Theorem \ref{closed general trees}, noting that
\[N(n,n-j)=N(n,j+1),\]
and
\[Q_a(n+1,n-j,3)=\sum_{s\geq 1}{n-2s-1\choose n-j+s+a-1}{n+2s-1\choose n-j-s-a-1}=\widetilde{Q}_a(n,j,3),\]
using the symmetry of the Narayana numbers and the binomial coefficients.

\end{proof}

\begin{remark}\label{identity anj1 remark}
If $j=0,1$, then $\widetilde{Q}_a(n,j,3)=0$ for $a=-1,0,1$, and equation \eqref{identity anj1 eq2} follows immediately.

However, for $j\geq 2$ there will be nonzero contributions from $\widetilde{Q}_a(n,j,3)$ for $a\leq j-3$. For example, if $j=2$ then 
\[\widetilde{Q}_{-1}(n,2,3)=\sum_{s\geq 1}{n-2s-1\choose 2-3s+1}{n+2s-1\choose 2+3s-1}={n+1\choose 4},\]
and the right hand side of \eqref{identity anj1 eq2} becomes
\[\frac{1}{n}{n\choose 3}{n\choose 2}-{n+1\choose 4}=
\frac{1}{n}{n\choose 3}{n\choose 2}-{n\choose 4}-{n\choose 3}\]
\[={n\choose 3}\left[\frac{n-1}{2}-1\right]-{n\choose 4}
=2{n\choose 3}\left[\frac{n-3}{4}\right]-{n\choose 4}={n\choose 4}\]
as expected.

\end{remark}

More generally, we can use Theorem \ref{anjk closed} when $k\geq 2$.

\begin{theorem}\label{identity anjk}
For $n\geq 1$, $j\geq 0$, and $k\geq 2$ we have
\[\sum_{j\geq i_{k-1}\geq \dots \geq i_1\geq 0}
\left(\prod_{m=0}^{k-1}{i_{m+2}+i_m\choose 2i_{m+1}}\right)\]
\[=N(n,j+1)-\left[\widetilde{Q}_1(n,j,k+2)-2\widetilde{Q}_0(n,j,k+2)+\widetilde{Q}_{-1}(n,j,k+2)\right],\]
where
\[N(n,j+1)=\frac{1}{n}{n\choose j+1}{n\choose j},\]
and
\[\widetilde{Q}_a(n,j,k+2)=\sum_{s\geq 1}{n-(k+1)s-1\choose j-(k+2)s-a}{n+(k+1)s-1\choose j+(k+2)s+a}.\]




\end{theorem}

\begin{proof}

From Corollary \ref{closed general perms} we have
\[a_{n,231,j}^{(k)}=h_{k+2}(n+1,n-j).\]
The left hand side of Theorem \ref{identity anjk} follows from Theorem \ref{anjk closed}. And the right hand side of Theorem \ref{identity anjk} follows from Theorem \ref{closed general trees}, noting that
\[Q_a(n+1,n-j,k+2)=\sum_{s\geq 1}{n-(k+1)s-1\choose n-j+s+a-1}{n+(k+1)s-1\choose n-j-s-a-1}\]
\[=\widetilde{Q}_a(n,j,k+2).\]

\end{proof}

\end{document}